\documentclass[a4paper]{article}

\usepackage{graphicx,graphics}
\usepackage{amsmath}
\usepackage{amsfonts}
\usepackage{enumerate}
\usepackage{tikz}
\usepackage{setspace}
\usepackage{subfigure}
\usepackage{caption}
\usepackage{xcolor}
\usepackage[ruled]{algorithm2e}
\makeatletter
\renewcommand{\@thesubfigure}{\hskip\subfiglabelskip}
\makeatother
\usepackage{float}
\usepackage{amsmath,arydshln}
\usepackage{epstopdf}

\def\@begintheorem#1#2{\par\bgroup{\scshape #1\ #2. }\it\ignorespaces}
\def\@opargbegintheorem#1#2#3{\par\bgroup%
   {\scshape #1\ #2\ ({\upshape #3}). }\it\ignorespaces}
\def\@endtheorem{\egroup}

\if@onethmnum
  \newtheorem{theorem}{Theorem}
  \newtheorem{lemma}[theorem]{Lemma}
  \newtheorem{corollary}[theorem]{Corollary}
  \newtheorem{proposition}[theorem]{Proposition}
  \newtheorem{definition}[theorem]{Definition}
  \newtheorem{remark}[theorem]{Remark}
  \newtheorem{example}[theorem]{Example}
\else
  \newtheorem{theorem}{Theorem}[section]
  \newtheorem{lemma}[theorem]{Lemma}

  \newtheorem{definition}[theorem]{Definition}
  \newtheorem{remark}[theorem]{Remark}
  
\fi

\title{Homotopy Methods for Eigenvector-Dependent Nonlinear Eigenvalue Problems
\thanks{The research was supported in part by
the National Natural Science Foundation of China (11971092)}}

\author{Xuping Zhang\thanks{School of Mathematical Sciences,
Dalian University of Technology, Dalian, Liaoning 116025, P. R.
China ({\tt zhangxp@dlut.edu.cn}).}
        \and Haimei Huo\thanks{School of Mathematical Sciences,
Dalian University of Technology, Dalian, Liaoning 116025, P. R.
China ({\tt ab1234@mail.dlut.edu.cn}).}}

\begin{document}

\maketitle

\begin{abstract}
Eigenvector-dependent nonlinear eigenvalue problems are considered which arise from the finite difference discretizations of the Gross-Pitaevskii equation. Existence and uniqueness of positive eigenvector for both one and two dimensional cases and existence of antisymmetric eigenvector for one dimensional case are proved. In order to compute eigenpairs corresponding to excited states as well as ground state, homotopies for both one and two dimensional problems are constructed respectively and the homotopy paths are proved to be regular and bounded. Numerical results are presented to verify the theories derived for both one and two dimensional problems.
\end{abstract}

%

{\bf Key Words} eigenvector-dependent nonlinear eigenvalue problem, Gross-Pitaevskii equation, homotopy continuation method

{\bf Subject Classification(AMS):}65H17, 65H20, 65N06, 65N25

\pagestyle{myheadings} \thispagestyle{plain} \markboth{XUPING ZHANG,
HAIMEI HUO}{HOMOTOPY METHODS FOR NONLINEAR EIGENVALUE PROBLEM}

\section{Introduction }
In this paper, we are concerned with the eigenvector-dependent nonlinear eigenvalue problems resulting from the finite difference discretizations of the Gross-Pitaevskii equation (GPE) describing Bose-Einstein condensates (BEC). BEC are clouds of ultracold alkali-metal atoms or molecules that occupy a single quantum state \cite{BT,CCJ}. The properties of a BEC at temperature $T$ much smaller than the critical condensation temperature $T_c$ are usually described by the nonlinear Schr\"{o}dinger equation (NLS) for the macroscopic wave function known as the Gross-Pitaevskii equation
\begin {equation} \label{sec 1:continuous nonlinear problem}
\begin{array}{lcl}
                {i \psi_t = -\frac{1}{2}\Delta \psi + V(x) \psi + \beta \ |\psi|^2 \psi},&& {t > 0 ,~x \in \Omega},\\
                {\psi(x,t) = 0},&& {t \ge 0,~x \in \partial \Omega},\end{array}\\
\end{equation}
where $ \psi = \psi(x,t)$ is the macroscopic wave function of the BEC, $V(x) = \frac{1}{2}(x_1^2+x_2^2+\cdots+x_N^2)$ is a typical trapping potential, $\Omega$ is a bounded domain in $\mathbb{R}^N$, $N \le 3$, and $\beta $ positive or negative corresponds to the defocusing or focusing NLS. Two important invariants of GPE are the
normalization of the wave function
\begin{equation}  \label{sec 1:constraint condition}
\begin{array}{lcl}
N(\psi) = \int_\Omega |\psi(x,t)|^2 dx = 1,&& t\ge 0,
\end{array}
\end{equation}
and the energy
\begin{align}  \label{sec 1:the-energy}
E(\psi(x,t)) = \int_\Omega \left[ \frac{1}{2}|\nabla \psi|^2 + V(x)|\psi|^2 + \frac{\beta}{2}|\psi|^4 \right] dx = E(\psi(x,0)).
\end{align}
To find stationary solution of (\ref{sec 1:continuous nonlinear problem}), we substitute the formula $\psi(x,t) = e^{-i \lambda t }\phi(x)$ into (\ref{sec 1:continuous nonlinear problem}) and (\ref{sec 1:constraint condition}) and obtain the time-independent Schr\"{o}dinger equation with Dirichlet boundary condition and the normalized condition
\begin{align}
\label{sec 1: transformed nonlinear problem}
\lambda\phi(x) & = -\frac{1}{2}\Delta \phi(x) + V(x)\phi(x) +\beta \phi^3(x),\quad {x\in \Omega},\\
\phi(x) & = 0, \quad x\in \partial \Omega,\\
\label{sec 1:transformed constraint}
\int_\Omega |\phi(x)|^2 dx & = 1,
\end{align}
where $\lambda$ is the chemical potential of the condensate and $\phi(x)$ is a real function independent of $t$ \cite{BD}. (\ref{sec 1: transformed nonlinear problem})-(\ref{sec 1:transformed constraint}) is a nonlinear eigenvalue problem. The eigenfunction corresponding to the minimum energy is called ground state and other eigenfunctions corresponding to larger energy are called excited states in the literature.

There have been many theoretical studies as well as numerical studies for the time-independent Schr\"{o}dinger equation. Bao and Cai \cite{B} pointed out when $\beta>0$, the positive ground state is unique, and if $V(x)$ is radially symmetric in 2D, the positive ground state must be radially symmetric. Bao and Tang \cite{BT} proposed methods by directly minimizing the energy functional via finite element approximation to obtain the ground state and by continuation method to obtain excited states. Edwards and Burnett \cite{MK} presented a Runge-Kutta type method and employed it to solve the spherically symmetric time-independent GPE. Adhikari \cite{A} used this approach to get the ground state solution of GPE in 2D with radial symmetry. Chang and Chien \cite{CC} and Chang, Chien and Jeng \cite{SCB} investigated stationary state solutions of $(\ref{sec 1: transformed nonlinear problem})$ using numerical continuation method, where $\lambda$ was treated as a continuation parameter. The solution curves branching from the first few bifurcation points of $(\ref{sec 1: transformed nonlinear problem})$ were numerically traced using continuation method under the normalization condition $(\ref{sec 1:transformed constraint})$.

Since nonlinearity rather than discretization method is our main concern and finite difference discretization will lead to a simpler nonlinear structure, finite difference discretization is adopted in this paper. The finite difference discretization of (\ref{sec 1: transformed nonlinear problem})-(\ref{sec 1:transformed constraint}) is the following eigenvector-dependent nonlinear eigenvalue problem,
\begin{equation}\label{sec 1: discretize nonlinear problem}
\begin{array}{c}
D\varphi + \beta \varphi^3 = \lambda\varphi,\\
h\varphi^{\mathrm{T}}\varphi - 1 = 0,
\end{array}
\end{equation}
where $D = \frac{1}{2}D_1 + V$, $D_1$ is the coefficient matrix corresponding to $-\Delta$, $V$ is the diagonal matrix corresponding to the potential $V(x)$, $\lambda$ and $\varphi$ are the unknowns, and $h$ is a constant related to mesh size. $\varphi^3$ represents the vector with elements being the corresponding elements of $\varphi$ to the power 3. This convention will be used throughout this paper.

With respect to the theoretical aspects of eigenvector-dependent nonlinear eigenvalue problem, \cite{CKMe} and \cite{CKM} studied the following general nonlinear eigen-value problem
\begin{equation}\label{6}
Ax + F(x) = \lambda x,
\end{equation}
where $A$ is an $n\times n$ irreducible Stieltjes matrix, i.e, an irreducible symmetric positive definite matrix with off-diagonal entries nonpositive, $F(x) = (f_1(x_1),\ldots,f_n(x_n))^\mathrm{T}$ and $x=(x_1,\cdots,x_n)^\mathrm{T}$. The functions $f_i(x_i)$ are assumed to have the property that $f_i(x_i) > 0$, when $x_i > 0$, $i=1,\ldots,n$. It is shown that under certain conditions on $F(x)$, there exists a positive eigenvector $x(\lambda)$ if and only if $\lambda > \mu $, where $ \mu $ is the smallest eigenvalue of $A$, and for every $\lambda > \mu $, the positive eigenvector is unique. Moreover, such a solution is a monotone increasing function of $\lambda$. The most popular numerical method to the eigenvector-dependent nonlinear eigenvalue problems is the self-consistent field (SCF) iteration, which is suitable for computing the ground state; for instance, see \cite{YLZ,SCF2} and the references therein. In \cite{JES}, inverse iteration method was applied to solve eigenvector-dependent nonlinear eigenvalue problems. Most of the above papers concentrate on the ground state and the first excited state. As far as we know, there are only a few numerical works on other excited states, such as \cite{BT,Yao-Zhou,Yang-Huang-Liu,Xie}. The main purpose of this paper is to design algorithms for computing excited states of high energy.

Homotopy method is one of the effective methods for solving eigenvalue problems. A great advantage of the homotopy method is that it is to a large degree parallel, in the sense that each eigenpath is traced independently of the others. There are several works on homotopy methods for linear eigenvalue problems. Remarkable numerical results have been obtained by using homotopy algorithm on eigenvalue problems of tridiagonal symmetric matrices \cite{TYA,TZS}. Solving eigenvalue problems of real nonsymmetric matrices with real homotopy was developed in \cite{LZC,Lui-Keller-Kwok}. The homotopy method is also used to solve the generalized eigenvalue problem \cite{TAT}. For eigenvalue-dependent nonlinear eigen-problems such as $\lambda$-matrix problems, a homotopy was given by Chu, Li and Sauer \cite{CT}.

The major part of this paper is the construction of homotopy for computing many eigenpairs of the eigenvector-dependent nonlinear eigen-problem. Key issues encountered in constructing the homotopy are the selection of the homotopy parameter and that of an appropriate initial eigenvalue problem so that the homotopy paths determined by the homotopy equation are regular and the numerical work in following these paths is at reasonable cost. The parameter $\beta$ in the original problem seems to be a natural choice for the homotopy parameter. However, it seems difficult to prove that 0 is a regular value for such homotopy. In fact, 0 is probably not a regular value of the natural homotopy with parameter $\beta$. Instead, an artificial parameter $t$ is chosen as the homotopy parameter to connect a constructed initial eigenvalue problem and the target one. As for the selection of an initial eigenvalue problem, random matrix with certain sparse structure is designed, which guarantees that 0 is a regular value of the homopoty with probability one and which renders the initial problem and the target problem possess similar structures.

The rest of this paper is organized as follows. In Section 2, the time-independent GPE Dirichlet problem (\ref{sec 1: transformed nonlinear problem})-(\ref{sec 1:transformed constraint}) is discretized by finite difference method and existence of certain types of solution of the discretized problems is derived. In Section 3, homotopies for (\ref{sec 1: discretize nonlinear problem}) are constructed with $\Omega\subset \mathbb{R}$ and $\Omega\subset \mathbb{R}^2$ respectively, and regularity and boundedness of the homotopy paths are proved. In Section 4, numerical results are presented to verify the theoretical results derived for $\Omega\subset \mathbb{R}$ and $\Omega\subset \mathbb{R}^2$ respectively. Conclusions are drawn in the last section.

\section{Discretizations of the nonlinear eigenvalue problem}

\subsection{Finite difference discretizations}
For one dimensional problem (\ref{sec 1: transformed nonlinear problem})-(\ref{sec 1:transformed constraint}) with $\Omega=[a,b]\subset \mathbb{R}$, the grid points are $x_j = a + jh$, $j = 0,\ldots, n+1$, where
$n\in \mathbb{N}^{+}$ and $h=\frac{b-a}{n+1}$ is the mesh size. The finite difference discretization of the differential equation and a simple quadrature of the normalization condition lead to the following system of algebraic equations,
\begin{equation}\label{sec 2: finite difference - one dimen}
\begin{array}{c}
D\varphi + \beta \varphi^3 - \lambda\varphi = 0,\\
\frac{1}{2} \left( \frac{1}{h} - \varphi^{\mathrm{T}}\varphi \right) = 0,
\end{array}
\end{equation}
where $\varphi=(\varphi_1,\cdots,\varphi_n)^{\mathrm{T}}$, $\varphi_j$ are the approximations of $\phi(x_j)$, $v_j = V(x_j)$, $j = 1,\ldots,n$, and $D = \frac{1}{2}D_1 + V$ with
\begin{align}
D_1=
\frac{1}{h^2}\left(
\begin{array}{cccc}
               2&-1&\\
                -1&2&\ddots&\\
           &\ddots&\ddots&-1\\
               &&-1&2
\end{array}\right), \quad
V=\left(
\begin{array}{ccc}
     v_1&&\\
     &\ddots&\\
     &&v_n
\end{array}\right).
\end{align}
The discretization of the normalized condition is rewritten so that the Jacobian matrix of the nonlinear mapping with respect to $(\varphi,\lambda)$ is symmetric, as will be seen below.

For two dimensional problem (\ref{sec 1: transformed nonlinear problem})-(\ref{sec 1:transformed constraint}) with $\Omega=[a,b]\times [c,d]$, the domain is divided into a $(m+1)\times(n+1)$ mesh with step size $h_1 = \frac{b-a}{m+1}$ in $x$-direction, $h_2 = \frac{d-c}{n+1}$ in $y$-direction. The grid points $(x_i,y_j)$ are
$x_i = a + ih_1$, $i = 0,\ldots,m+1$, and $y_j = c + jh_2$, $j = 0,\ldots,n+1$. Using central difference, we get
\begin{equation}\label{sec 2: finite difference - two dimen}
\begin{array}{c}
D\varphi + \beta \varphi^3 - \lambda\varphi = 0,\\
\frac{1}{2} \left( \frac{1}{h_1 h_2} - \varphi^{\mathrm{T}}\varphi \right) = 0,
\end{array}
\end{equation}
where $\varphi=(\varphi_{11},\cdots,\varphi_{1n},\cdots,\varphi_{m1},\cdots,\varphi_{mn})^{\mathrm{T}}$, $\varphi_{ij}$ are the approximations of $\phi(x_i,y_j)$, $v_{ij} = V(x_i,y_j)$, $i=1,\ldots,m$, $j = 1,\ldots,n$ and D is a block tridiagonal matrix
\begin{equation*}
D=
 \setlength{\arraycolsep}{2.5pt}
 \renewcommand{\arraystretch}{0.8}
\frac{1}{2}\left(
\begin{array}{cccc}
      D_{11} & D_{12} &\\
      D_{21} & D_{22} &\ddots&\\
      &\ddots&\ddots& D_{{m-1}, m}\\
     && D_{m, {m-1}} & D_{mm}
\end{array}
\right)
+
\left(
\begin{array}{cccc}
      V_{1} & &\\
      & V_{2}&&\\
      &&\ddots&\\
     &&&V_{m}
\end{array}
\right),
\end{equation*}
where
\begin{align}
D_{i i}=
\left(
\begin{array}{cccc}
    \frac{2}{h_1^2}+ \frac{2}{h_2^2}&-\frac{1}{h_2^2}&\\
     -\frac{1}{h_2^2}&\frac{2}{h_1^2}+\frac{2}{h_2^2}&\ddots&\\
      &\ddots&\ddots&-\frac{1}{h_2^2}\\
     &&-\frac{1}{h_2^2}& \frac{2}{h_1^2}+\frac{2}{h_2^2}
\end{array}\right) \in \mathbb{R}^{n \times n},
\end{align}

\begin{align}
D_{i-1, i}=
\left(
\begin{array}{cccc}
            -\frac{1}{h_1^2}&&&\\
            &-\frac{1}{h_1^2}&&\\
            &&\ddots&\\
            &&&-\frac{1}{h_1^2}
\end{array}\right)\in \mathbb{R}^{n \times n}, \quad D_{{i-1}, i}=D_{i, {i-1}},
\end{align}

\begin{align}
V_{i}=
\left(
\begin{array}{cccc}
            v_{i1}&&&\\
            &v_{i2}&&\\
            &&\ddots&\\
            &&&v_{in}
\end{array}\right), \quad i = 1,\ldots,m.
\end{align}

\begin{remark} \label{sec 2: differential-matrix-positive-definite-1D2D}
For both one and two dimensional cases, D is an irreducible symmetric diagonal dominant matrix and the diagonal entries of D are all positive. Therefore D is positive definite.
\end{remark}

\subsection{Existence of certain types of solution}

In this subsection, we will study the existence of solution for the discretized nonlinear eigenvalue problem. From \cite{B}, we know for (\ref{sec 1: transformed nonlinear problem})-(\ref{sec 1:transformed constraint}), when $\beta>0$, the positive ground state is unique, and if $V(x)$ is radially symmetric in 2D, the positive ground state must be radially symmetric. We will prove the existence of positive solution and the existence of antisymmetric solution for discretized nonlinear eigenvalue problem. For convenient reading, two underlying theorems from \cite{CKMe} are quoted as underlying lemmas.

\begin{lemma}(\cite{CKMe})
\label{sec 2: existence of solution}
Let A be an irreducible Stieltjes matrix and $\mu$ be the smallest positive eigenvalue of A. Let $\lambda > \mu$ and let
\begin {equation}
F(x) = \left(\begin{array}{c}
       f_1(x_1)\\
       \vdots \\
       f_n(x_n)\end{array}\right),
\end{equation}
where for $i = 1,\ldots,n$, $f_i(x):[0,\infty)\rightarrow[0,\infty)$ are $C^1$ functions satisfying the conditions:
\begin{equation}\label{sec 2: limit conditon}
\lim\limits_{t\rightarrow 0 }\frac{f_i(t)}{t} = 0,\qquad \lim\limits_{t\rightarrow \infty }\frac{f_i(t)}{t} = \infty.
\end{equation}
Then $Ax + F(x) = \lambda x$ has a positive solution. If, in addition, for $i = 1,\ldots,n$,
\begin{equation} \label{sec 2: inequality condition}
\frac{f_i(s)}{s} < \frac{f_i(t)}{t}
\end{equation}
whenever $0< s < t$, then the solution is unique.
\end{lemma}

\begin{lemma}(\cite{CKMe})
\label{sec2: normalization solution}
Let the conditions $(\ref{sec 2: limit conditon})$ and $(\ref{sec 2: inequality condition})$ of Lemma $\ref{sec 2: existence of solution}$ be satisfied and let $x(\lambda)$ denote the unique positive eigenvector corresponding
to $\lambda \in (\mu,\infty)$. Then:
\begin{enumerate}[(i)]
\item $x(\lambda_1) < x(\lambda_2)$, if $\mu < \lambda_1 < \lambda_2 < \infty$;
\item $x(\lambda)$ is continuous on $(\mu,\infty)$;
\item $\lim\limits_{\lambda \rightarrow \infty }{x_i(\lambda)} = \infty$, $i=1,\cdots,n$;
\item $\lim\limits_{\lambda \rightarrow \mu^{+} }{x_i(\lambda)} = 0$, $i=1,\cdots,n$.
\end{enumerate}
\end{lemma}

\begin{remark} \label{sec 2: unique normalization solution}
Lemma $\ref{sec2: normalization solution}$ indicates for any given normalization $r > 0$, there exist a unique $\lambda > \mu$ and unique positive $x(\lambda)$
such that $||x(\lambda)|| = r$.
\end{remark}

\begin{theorem}\label{sec 2: thm-positive-eigenvector}
If $\beta > 0$, there exist unique positive eigenvectors for problem $(\ref{sec 2: finite difference - one dimen})$ and $(\ref{sec 2: finite difference - two dimen})$ respectively.
\end{theorem}

\begin{proof}
From Remark \ref{sec 2: differential-matrix-positive-definite-1D2D}, we know $D$ in both $(\ref{sec 2: finite difference - one dimen})$ and $(\ref{sec 2: finite difference - two dimen})$ is an irreducible Stieltjes matrix. In addition it can be verified that $\beta\varphi^3$ in both $(\ref{sec 2: finite difference - one dimen})$ and $(\ref{sec 2: finite difference - two dimen})$ satisfies the conditions of $F(x)$ in Lemma $\ref{sec 2: existence of solution}$. From Remark \ref{sec 2: unique normalization solution}, the claim is proved.
\end{proof}

\begin{theorem} \label{sec 2: certain solution}
Let $\beta > 0$ and $\Omega=[-a,a]$ for problem $(\ref{sec 2: finite difference - one dimen})$. 
\begin{enumerate}[(i)]
\item When $n$ is odd and the grid points $x_i$, $i=1,\ldots,n$, satisfy
\begin{equation}
\label{sec 2: odd condition}
\nonumber
x_1 = - x_n,~x_2 = - x_{n-1},~\cdots,~x_{\frac{n-1}{2}} = - x_{\frac{n+3}{2}},~x_{\frac{n+1}{2}} = 0,
\end{equation}
there exists a unique solution $\varphi = (\varphi_1,\varphi_2,\cdots,\varphi_n)^{\mathrm{T}}$ with $\varphi_1 = -\varphi_n$, $\varphi_2 = - \varphi_{n-1}, \cdots$, $\varphi_{\frac{n-1}{2}} = - \varphi_{\frac{n+3}{2}}$, $\varphi_{\frac{n+1}{2}} = 0$, $\varphi_j > 0$, $j=1,\ldots,n$.
\item When $n$ is even and the grid points $x_i$, $i=1,\ldots,n$, satisfy
\begin{equation}
\label{sec 2: even condition}
\nonumber
x_1 = - x_n,~x_2 = - x_{n-1},~\cdots,~x_{\frac{n}{2}} = - x_{\frac{n}{2}+1},
\end{equation}
there exists a unique solution $\varphi = (\varphi_1,\varphi_2,\cdots,\varphi_n)^{\mathrm{T}}$ with $\varphi_1 = -\varphi_n$, $\varphi_2 = - \varphi_{n-1}$, $\cdots$, $\varphi_{\frac{n}{2}} = - \varphi_{\frac{n}{2}+1}$, $\varphi_j > 0$, $j=1,\ldots,n$.
\end{enumerate}
\end{theorem}

\begin{proof}
(i). When $n$ is odd, consider the following equations
\begin{equation}
\label{sec2:half-equations-n-odd}
\begin{array}{c}
D_2\varphi + \beta \varphi^3 = \lambda\varphi,\\
\varphi^\mathrm{T} \varphi -\frac{1}{2h} = 0,
\end{array}
\end{equation}
where
\begin{align}
D_2=
\left(\begin{array}{cccc}
      \frac{1}{h^2}+\frac{1}{2}x_1^2&-\frac{1}{2h^2}&&\\
      -\frac{1}{2h^2} &\frac{1}{h^2}+\frac{1}{2}x_2^2& \ddots&\\
       &\ddots&\ddots&-\frac{1}{2h^2}\\
        &&-\frac{1}{2h^2}&\frac{1}{h^2}+\frac{1}{2}x_{\frac{n-1}{2}}^2 \end{array}\right)\in \mathbb{R}^{\frac{n-1}{2} \times \frac{n-1}{2}}.
\end{align}
Note that $D_2$ is an irreducible Stieltjes matrix and $\beta \varphi^3$ satisfies the conditions of $F(x)$ in Lemma $\ref{sec 2: existence of solution}$. Therefore there exists a unique positive solution $\varphi = (\varphi_1,\varphi_2,\ldots,\varphi_{\frac{n-1}{2}})^{\mathrm{T}}$ for (\ref{sec2:half-equations-n-odd}). Due to the relations $x_1 = - x_n,~x_2 = - x_{n-1},~\cdots,~x_{\frac{n-1}{2}} = - x_{\frac{n+3}{2}},~x_{\frac{n+1}{2}} = 0$, set $\varphi_n = -\varphi_1$, $\varphi_{n-1} = - \varphi_2$, $\cdots$,
$\varphi_{\frac{n+3}{2}} = - \varphi_{\frac{n-1}{2}}$, $\varphi_{\frac{n+1}{2}} = 0$. Then $\varphi = (\varphi_1,\varphi_2,\ldots,\varphi_n)^{\mathrm{T}}$ is a solution of $(\ref{sec 2: finite difference - one dimen})$.

(ii). When $n$ is even, consider the following equations
\begin{equation}
\label{sec2:half-equations-n-even}
\begin{array}{c}
D_2\varphi + \beta \varphi^3 = \lambda\varphi,\\
\varphi^{\mathrm{T}}\varphi -\frac{1}{2h} = 0,
\end{array}
\end{equation}
where
\begin{align}
D_2=
\left(\begin{array}{cccc}
      \frac{1}{h^2}+\frac{1}{2}x_1^2&-\frac{1}{2h^2}&&\\
      -\frac{1}{2h^2} &\frac{1}{h^2}+\frac{1}{2}x_2^2&\ddots&\\
        &\ddots&\ddots&-\frac{1}{2h^2}\\
        &&-\frac{1}{2h^2}&\frac{3}{2h^2}+\frac{1}{2}x_{\frac{n}{2}}^2 \end{array}\right) \in \mathbb{R}^{\frac{n}{2} \times \frac{n}{2}}.
\end{align}
Similarly there exists a unique positive solution $\varphi = (\varphi_1,\varphi_2,\ldots,\varphi_{\frac{n}{2}})^{\mathrm{T}}$ for (\ref{sec2:half-equations-n-even}). Due to the relations $x_1 = - x_n,~x_2 = - x_{n-1},~\cdots,~x_{\frac{n}{2}} = - x_{\frac{n}{2}+1}$, set $\varphi_n = - \varphi_1$, $\varphi_{n-1} = - \varphi_2$, $\cdots$, $\varphi_{\frac{n}{2}+1} = - \varphi_{\frac{n}{2}}$. Then $\varphi = (\varphi_1,\varphi_2,\ldots,\varphi_n)^{\mathrm{T}}$ is a solution of $(\ref{sec 2: finite difference - one dimen})$.
\end{proof}

\section{Homotopy methods}
In this section, in order to compute many eigenpairs, we construct homotopy equations  for 1D discretized problem $(\ref{sec 2: finite difference - one dimen})$ and 2D discretized problem $(\ref{sec 2: finite difference - two dimen})$ respectively. We shall prove the regularity and boundedness of the homotopy paths. The regularity of homotopy paths can be usually obtained by random perturbations of appropriate parameters, so the most important feature of our construction is the choice of appropriate parameters. In addition, if the initial matrix can be chosen as close to the matrix $D$ as possible, then most of the homotopy paths are close to straight lines and will be easy to follow \cite{LZC}.

\subsection{One dimensional case}
For $(\ref{sec 2: finite difference - one dimen})$, a homotopy $H:\mathbb{R}^n\times \mathbb{R}\times [0,1] \rightarrow \mathbb{R}^n\times \mathbb{R}$ is defined as
\begin{equation}\label{sec 3: homotopy one dimen}
H(\varphi,\lambda,t)=
\left(
\begin{array}{c}
     (1-t)A(K)\varphi+D\varphi+t\beta\varphi^3-\lambda\varphi\\
      \frac{1}{2} \left( \frac{1}{h} - \varphi^{\mathrm{T}}\varphi \right)
\end{array}
\right)=0,
\end{equation}
where $A(K) = \mbox{diag}(a_1,\cdots,a_n)$ is a random diagonal matrix with $K =(a_1,\cdots,a_n)^{\mathrm{T}} \in \mathbb{R}^n$. At $t = 0$,
$H(\varphi,\lambda,0)$ corresponds to the linear eigenvalue problem
\begin{equation}
H(\varphi,\lambda,0)=
\left(
\begin{array}{c}
      A(K)\varphi+D\varphi-\lambda\varphi\\
      \frac{1}{2} \left( \frac{1}{h} - \varphi^{\mathrm{T}}\varphi \right)
\end{array}
\right)=0,
\end{equation}
while at $t = 1$, $H(\varphi,\lambda,1) = 0$ corresponds to the problem (\ref{sec 2: finite difference - one dimen}). Assuming that the eigenpairs of $H(\varphi,\lambda,0) = 0$ are $(\varphi^{(i)},\lambda_i)$, $i=1,\ldots,n$, $\lambda_1 \le \cdots \le \lambda_n$, we shall use these $n$ points $(\varphi^{(i)},\lambda_i,0)$ as our initial points when tracing the homotopy curves leading to the desired solutions of $H(\varphi,\lambda,1) = 0$.

The choice of the initial matrix $A(K) + D$ provides some advantages. First, it makes sure that $\forall t\in[0,1)$, what we need to solve is a sparse nonlinear eigenvalue problem with constraint. Second, since $\forall K \in \mathbb{R}^n$, $\forall t\in[0,1)$ and $\varphi \in \mathbb{R}^n$, $(1-t)A(K) + D + t\beta \mbox{diag}(\varphi^2)$ is a real symmetric matrix, the solution curves starting from the initial points at $t = 0$ are real. Finally, since $A(K)+ D$ is a tridiagonal matrix with all the subdiagonal and supdiagonal entries nonzero, all the eigenvalues of $A(K)+D$ are simple and the Jacobian matrix of $H$ at $(\varphi_0,\lambda_0,0)$ is nonsingular for $\forall(\varphi_0,\lambda_0)$ such that $H(\varphi_0,\lambda_0,0) = 0$. Thus locally a unique curve around $(\varphi_0,\lambda_0,0)$ is guaranteed.

The effectiveness of the homotopy is based on the following Parametrized Sard's Theorem.

\begin{theorem}(Parametrized Sard's Theorem)\label{sec 3:Parametrized Sard's Theorem}
Let $f:M \times P \subset \mathbb{R}^{m} \times \mathbb{R}^{q} \to \mathbb{R}^{n} $ be a $C^{k}$ mapping with $k > max(0,m-n)$, where $M$ and $P$ are open sets in $\mathbb{R}^{m}$ and $\mathbb{R}^{q}$ respectively. If $y$ is a regular value of f, then
y is also a regular value of $f(\cdot,p)$ for almost all $p \in P$.
\end{theorem}

In the rest of this paper, we will denote the $i$-th row of a matrix $M$ as $M(i,:)$ and the $j$-th column of $M$ as $M(:,j)$. If $I_1$ is a row index set, $M(I_1,:)$ will be the submatrix formed by the $I_1$ rows of $M$. $M(i:end,:)$ will be the submatrix formed by the rows from the $i$-th row to the last. Similarly $M(:,J_1)$ will be the submatrix formed by the $J_1$ columns of $M$. If $I_1$ and $I_2$ are two index sets, $[M(I_1,:);M(I_2,:)]$ denotes the submatrix formed by the $I_1$ rows and the $I_2$ rows of $M$. In Theorem \ref{sec 3: thm-regular one dimen}, we prove regularity and boundedness of the homotopy paths determined by the homotopy equation (\ref{sec 3: homotopy one dimen}).

\begin{theorem}
\label{sec 3: thm-regular one dimen}
For the homotopy $H:\mathbb{R}^n\times \mathbb{R}\times [0,1) \rightarrow \mathbb{R}^n\times \mathbb{R}$ in (\ref{sec 3: homotopy one dimen}), for almost all  $K\in \mathbb{R}^n$,
\begin{enumerate}[(i)]
\item 0 is a regular value of $H$ and therefore the homotopy curves corresponding to different initial points do not intersect each other for $t\in[0,1)$;
\item Every homotopy path $(\varphi(s),\lambda(s),t(s)) \subset H^{-1}(0)$ is bounded.
\end{enumerate}
\end{theorem}

\begin{proof}
(i) Define a mapping $\tilde{H}: \mathbb{R}^n\times \mathbb{R}\times [0,1]\times \mathbb{R}^n \rightarrow \mathbb{R}^n\times \mathbb{R}$ related to $H$ as follows
\begin{equation}\label{sec3: mapping-related-to-homotopy-1D}
\tilde{H} (\varphi,\lambda,t,K)=
\left(
\begin{array}{c}
     (1-t)A(K)\varphi+D\varphi+t\beta\varphi^3-\lambda\varphi\\
      \frac{1}{2} \left( \frac{1}{h} - \varphi^{\mathrm{T}}\varphi \right)
\end{array}
\right)=0,
\end{equation}
such that $H(\varphi,\lambda,t) = \tilde{H} (\varphi,\lambda,t,K)$. The Jacobian matrix of $\tilde{H}$, $\frac{\partial \tilde{H}}{\partial(\varphi,\lambda,t,K)}$, is
\begin{equation*}
\left(
\begin{array}{cccc}
(1-t)A(K)+D+3t\beta \mbox{diag}(\varphi^2)-\lambda I &-\varphi& \beta\varphi^3-A(K)\varphi&(1-t)\mbox{diag}(\varphi)\\
                                              -\varphi^{\mathrm{T}}&0&0&0
\end{array}
\right).
\end{equation*}
Divide $\{1,\cdots,n\}$ into two parts $C^0$ and $C^*$, where $C^0$ denotes the indices $i$ such that $\varphi_i = 0$ and $C^*$ denotes the indices $i$ such that $\varphi_i \neq 0$. Since $\forall i \in C^0$, both the columns $3t\beta \mbox{diag}(\varphi^2)(:,i)$ and $t\beta \mbox{diag}(\varphi^2)(:,i)$ are equal to zero, it holds that,
\begin{align}
\nonumber
& \left( (1-t)A(K)+D+3t\beta \mbox{diag}(\varphi^2)-\lambda I \right)(:,C^0) \\
= & \left( (1-t)A(K)+D+t\beta \mbox{diag}(\varphi^2)-\lambda I \right)(:,C^0).
\end{align}
For the columns in $C^*$, the diagonal matrix $\mbox{diag}(\varphi)$ is nonzero. By elementary column transformations, the Jacobian matrix $\frac{\partial \tilde{H}}{\partial(\varphi,\lambda,t,K)}$ is transformed to the following,
\begin{equation*}
\left(
\begin{array}{cccc}
(1-t)A(K)+D+t\beta \mbox{diag}(\varphi^2)-\lambda I &-\varphi& \beta\varphi^3-A(K)\varphi&(1-t)\mbox{diag}(\varphi)\\
                                              -\varphi^{\mathrm{T}}&0&0&0
\end{array}
\right).
\end{equation*}
Define
\begin{align}
F_1 =
\left(
\begin{array}{cc}
(1-t)A(K) + D + t\beta \mbox{diag}(\varphi^2)-\lambda I & -\varphi\\
-\varphi^{\mathrm{T}} & 0
\end{array}
\right).
\end{align}
For $\forall K \in \mathbb{R}^n$, $\forall (\varphi,\lambda,t)\in \mathbb{R}^n \times R \times [0,1)$ satisfying $(\ref{sec3: mapping-related-to-homotopy-1D})$, since the subdiagonals and supdiagonals of the matrix $(1-t)A(K) + D + t \beta \mbox{diag}(\varphi^2)$ are nonzero, $\lambda$ is a simple eigenvalue of $(1-t)A(K) + D + t \beta \mbox{diag}(\varphi^2)$. Therefore $F_1$ is nonsingular and $\frac{\partial \tilde{H}}{\partial(\varphi,\lambda,t,K)}$ is row full rank. As a result, 0 is a regular value of the mapping $\tilde{H}$. From Theorem \ref{sec 3:Parametrized Sard's Theorem}, for almost all $K\in \mathbb{R}^n$, 0 is a regular value of the restricted mapping $\tilde{H}(\cdot,\cdot,\cdot,K)$, i.e., $H$.

(ii) From $(\ref{sec 3: homotopy one dimen})$, for fixed $K \in \mathbb{R}^n$, $\forall (\varphi,\lambda, t) \in \mathbb{R}^n \times \mathbb{R} \times [0,1)$ satisfying $H(\varphi,\lambda,t)$$ = $$0$, $\|\varphi\| = 1/\sqrt{h}$ and $\lambda = h \left( (1-t)\varphi^{\mathrm{T}} A(K) \varphi + \varphi^{\mathrm{T}}D\varphi + t\beta\varphi^{\mathrm{T}}\varphi^3 \right)$. Then
\begin{eqnarray*}
|\lambda| &= &h|(1-t)\varphi^{\mathrm{T}} A(K)\varphi + \varphi^{\mathrm{T}}D\varphi + t\beta\varphi^{\mathrm{T}}\varphi^3| \nonumber \\
         & \le&(1-t)\rho(A(K))+\rho(D)+\frac{t\beta}{h} \nonumber \\
        &   \le&\rho(A(K))+\rho(D)+\frac{\beta}{h},
\end{eqnarray*}
where $\rho(M)$ denotes the spectral radius of $M$.
\end{proof}

\subsection{Two dimensional case}
\subsubsection{The homotopy with random tridiagonal matrix}
\label{sec3:subsubset-tridiagonal-case}
For two dimensional case, the homotopy $H:\mathbb{R}^{mn}\times \mathbb{R}\times [0,1] \rightarrow \mathbb{R}^{mn}\times \mathbb{R}$ is constructed as follows
\begin {equation}
\label{sec3:eqn-homotopy-3diagonal}
H(\varphi,\lambda,t)=
\left(
\begin{array}{c}
     (1-t)A(K)\varphi+D\varphi+t\beta\varphi^3-\lambda\varphi\\
     \frac{1}{2} \left( \frac{1}{h_1 h_2} - \varphi^{\mathrm{T}}\varphi \right)
\end{array}
\right)=0,
\end{equation}
where
$A(K) \in \mathbb{R}^{mn \times mn}$ is a random block diagonal matrix with tridiagonal blocks, namely,
\begin{equation*}
A(K)=
\setlength{\arraycolsep}{0.5pt}
 \renewcommand{\arraystretch}{0.1}
\left(\begin{array}{cccc}
    A_{1}&&&\\
    &A_{2}&&\\
    &&\ddots&\\
    &&&A_{m}\end{array}\right) \mbox{with} ~A_{i}=
                          \setlength{\arraycolsep}{2pt}
                            \renewcommand{\arraystretch}{0.5}
                              \left(\begin{array}{ccccc}
                            a_{11}^{(i)} & a_{12}^{(i)}&&&\\
                         a_{12}^{(i)} & a_{22}^{(i)} & a_{23}^{(i)} &&\\
                           & a_{23}^{(i)} & a_{33}^{(i)} &\ddots&\\
                             && \ddots & \ddots & a_{n-1,n}^{(i)}\\
                              &&& a_{n-1,n}^{(i)} & a_{nn}^{(i)}\end{array}\right),~i=1,\ldots,m,
\end{equation*}
and
$K =\left( K_1,K_2,\ldots,K_m \right)^{\mathrm{T}}$ with
$K_i =\left( a_{11}^{(i)},a_{12}^{(i)},a_{22}^{(i)},a_{23}^{(i)},\ldots,a_{n-1,n-1}^{(i)},a_{n-1,n}^{(i)},a_{nn}^{(i)} \right)$. $H(\varphi,\lambda,0)$ corresponds to the linear eigenvalue problem
\begin{equation*}
H(\varphi,\lambda,0)=
\left(
\begin{array}{c}
      A(K)\varphi+D\varphi-\lambda\varphi\\
      \frac{1}{2} \left( \frac{1}{h_1 h_2} - \varphi^{\mathrm{T}}\varphi \right)
\end{array}
\right)=0,
\end{equation*}
while $H(\varphi,\lambda,1)$ corresponds to the problem $(\ref{sec 2: finite difference - two dimen})$.

In order to show the effectiveness of this homotopy $H$ by the Parametrized Sard's Theorem, define a mapping $\tilde{H}: \mathbb{R}^{mn}\times \mathbb{R}\times [0,1] \times \mathbb{R}^{(2nm-m)} \rightarrow \mathbb{R}^{mn}\times \mathbb{R}$ related to $H$ as follows
\begin{equation}\label{sec3: mapping-related-to-homotopy-2D}
\tilde{H} (\varphi,\lambda,t,K)=
\left(
\begin{array}{c}
     (1-t)A(K)\varphi+D\varphi+t\beta\varphi^3-\lambda\varphi\\
      \frac{1}{2} \left( \frac{1}{h_1 h_2} - \varphi^{\mathrm{T}}\varphi \right)
\end{array}
\right)=0,
\end{equation}
such that $H(\varphi,\lambda,t) = \tilde{H} (\varphi,\lambda,t,K)$. The Jacobian matrix of $\tilde{H}$, $\frac{\partial \tilde{H}}{\partial(\varphi,\lambda,t,K)}$, is
\begin{equation}
\label{sec3:jacobi-matrix-for-tridiagonal}
\setlength{\arraycolsep}{8pt}
\renewcommand{\arraystretch}{0.45}
\left(
\begin{array}{cccc}
(1-t)A(K)+D+3t\beta \mbox{diag}(\varphi^2)-\lambda I &-\varphi& \beta\varphi^3-A(K)\varphi&(1-t)B\\
                                              -\varphi^{\mathrm{T}}&0&0&0
\end{array}
\right),
\end{equation}
where $B=\frac{\partial (A(K)\varphi)}{\partial K} \in \mathbb{R}^{{(mn)}\times {(2nm-m)}}$. Denote $\varphi_{i} = \left( \varphi_{i1},\varphi_{i2},\ldots,\varphi_{in} \right)^{\mathrm{T}}$. It can be verified that
\begin{equation}
B = \left(
  \begin{array}{cccc}
     B_{1}&&&\\
    &B_{2}&&\\
    &&\ddots&\\
    &&&B_{m}
  \end{array}
\right)
\label{sec3:def-B}
\end{equation}
with
\begin{equation}
B_i = \frac{\partial (A_i \varphi_i)}{\partial K_i}
= \left(
  \begin{array}{ccccccccccc}
     \varphi_{i1} & \varphi_{i2} & & & & & & & & &\\
    & \varphi_{i1} & \varphi_{i2} & \varphi_{i3} & & & & & & &\\
    & & & \varphi_{i2} & \varphi_{i3} & \varphi_{i4} & & & & &\\
    & & & & & & \ddots & & & &\\
    & & & & & & & \varphi_{i,n-2} & \varphi_{i,n-1} & \varphi_{in} &\\
    & & & & & & & & & \varphi_{i,n-1} & \varphi_{in}
  \end{array}
\right).
\end{equation}
For example, when $m = 2,n = 3$, we have
\begin{equation*}
\setcounter{MaxMatrixCols}{14}
B =
\setlength{\arraycolsep}{7.5pt}
\renewcommand{\arraystretch}{1}
 \begin{pmatrix}
  \varphi_{11} & \varphi_{12} &{0}&{0}&{0}&{0}&{0}&{0}&{0}&{0}\\
  {0} & \varphi_{11}&\varphi_{12}&\varphi_{13}&{0}&{0}&{0}&{0}&{0}&{0}\\
  {0} &{0}&{0}& \varphi_{12}&\varphi_{13}&{0}&{0}&{0}&{0}&{0}\\
  {0} &{0}&{0}&{0}&{0}&\varphi_{21}&\varphi_{22}&{0}&{0}&{0}\\
  {0} &{0}&{0}&{0}&{0}& {0}&\varphi_{21}&\varphi_{22}&\varphi_{23}&{0}\\
  {0} &{0}&{0}&{0}&{0}&{0}&{0}&{0}&\varphi_{22}&\varphi_{23}
\end{pmatrix}.
\end{equation*}

Denote by $G_a$ all the inner grid points and by $R_a$ the ordering of the grid points in $G_a$, i.e.,
\begin{align}
G_a=\{ (i,j):i=1,\ldots,m,j=1,\ldots,n \},\\
R_a=\{ j + (i-1)*n:(i,j)\in G_a \},
\end{align}
and by $G^0$ the grid points with function value being zero,
\begin{align}
G^0=\{ (i,j)\in G_a:\varphi_{ij}=0 \}.
\end{align}
Note that $G_a$ and $R_a$ have a one to one correspondence. Define such correspondence as a mapping $\Gamma : G_a \rightarrow R_a$,
\begin{align}
\Gamma(i,j) = j + (i-1)*n.
\end{align}
Denote by $R^0$ the indices of rows in which $B$ is zero, by $R^*$ the indices of rows, in which $B$ is not zero, and $S_i^0$ and $S_i^*$ with similar meanings for $B_i$,
\begin{align}
\label{sec3:notation-R0-R*}
R^0=\{ r:B(r,:)=0 \}, \quad R^*=\{ r:B(r,:)\neq 0 \},\\
S^0_i=\{ r:B_i(r,:)=0 \}, \quad S_i^*=\{ r:B_i(r,:)\neq 0 \}.
\end{align}
It can be verified that
\begin{align}
R^0 = \bigcup_{i=1}^{m}S^0_i, \quad R^* = \bigcup_{i=1}^{m}S_i^*.
\end{align}

\begin{lemma}
\label{sec3:lemma-row-full-rank-BR1}
Let $B$ be the matrix defined in (\ref{sec3:def-B}). Then the nonzero rows of $B$ are linearly independent, i.e., the submatrix $B(R^*,:)$ is row full rank.
\end{lemma}

\begin{proof}
Note that $B(R^*,:) = [B(S_1^*,:);\ldots; B(S_m^*,:)]$. Due to the block structure of $B$, it suffices to prove that for any $i$, $1 \le i \le m$, $B_i(S_i^*,:)$ is row full rank if $S_i^*$ is not empty.

Now suppose that $S_i^*$ is not empty, that is $\varphi_i \neq 0$. Let the nonzero components of $\varphi_i$ be $\varphi_{i,i_1},\varphi_{i,i_2},\ldots,\varphi_{i,i_r}$, with $i_1 < i_2 < \ldots < i_r$. Denote by $Z(\varphi_{i,i_1},\ldots,\varphi_{i,i_k})$ the rows of $B_i$ containing $\varphi_{i,i_1},\varphi_{i,i_2},\ldots,\varphi_{i,i_k}$, $1 \leq k \leq r$. Then $B_i(S_i^*,:) = Z(\varphi_{i,i_1},\ldots,\varphi_{i,i_r})$. The claim will be proved by successively adding rows with nonzero component of $\varphi_{i}$ to $Z$.

(1) Prove $Z(\varphi_{i,i_1})$ is row full rank. If $i_1 = 1$, $Z(\varphi_{i,i_1})$ is the first two rows of $B_i$ and it is row full rank. If $1 < i_1 < n$, $Z(\varphi_{i,i_1})$ is three successive rows of $B_i$ which contains the following submatrix involving  $\varphi_{i,i_1}$,
\begin{equation}
\left(
  \begin{array}{ccc}
     \varphi_{i,i_1} & 0 & 0\\
    * & \varphi_{i,i_1} & *\\
    0 & 0 & \varphi_{i,i_1}
  \end{array}
\right) \nonumber
\end{equation}
where $*$ stands for an element which may be zero or nonzero. Therefore $Z(\varphi_{i,i_1})$ is row full rank. If $i_1 = n$, $Z(\varphi_{i,i_1})$ is the last two rows of $B_i$ and is row full rank too.

(2) Prove that when $Z(\varphi_{i,i_1},\ldots,\varphi_{i,i_k})$ is row full rank, $Z(\varphi_{i,i_1},\ldots,\varphi_{i,i_{k+1}})$ is also row full rank, $1 \leq k < r$. If $i_k = n-1$, then $i_{k+1} = n$ and $Z(\varphi_{i,i_1},\ldots,\varphi_{i,i_{k+1}})=Z(\varphi_{i,i_1},\ldots,\varphi_{i,i_k})$. Next suppose $i_k < n-1$. If $i_{k+1} = i_k + 1$, i.e., $\varphi_{i,i_k}$ and $\varphi_{i,i_{k+1}}$ are successive, then $Z(\varphi_{i,i_1},\ldots,\varphi_{i,i_{k+1}})$ consists of $Z(\varphi_{i,i_1},\ldots,\varphi_{i,i_k})$ and a new row with $\varphi_{i,i_{k+1}}$ as its first nonzero element and with $\varphi_{i,i_{k+1}}$ located in a column different from those of $\varphi_{i,i_1},\ldots,\varphi_{i,i_k}$, the connecting submatrix illustrated in the following,
\begin{equation}
\left(
  \begin{array}{cccccc}
   * & \varphi_{i,i_k} & \varphi_{i,i_{k+1}} & 0 & 0 & 0\\
   0 & 0 & \varphi_{i,i_k} & \varphi_{i,i_{k+1}} & * & 0\\
   0 & 0 & 0 & 0& \varphi_{i,i_{k+1}} & *
  \end{array}
\right) \nonumber
\end{equation}
Thus $Z(\varphi_{i,i_1},\ldots,\varphi_{i,i_{k+1}})$ is row full rank. If $i_{k+1} = i_k + 2 < n$, then $Z(\varphi_{i,i_1},\ldots,\varphi_{i,i_{k+1}})$ consists of $Z(\varphi_{i,i_1},\ldots,\varphi_{i,i_k})$ and two new rows and similarly these two new rows are linearly independent with $Z(\varphi_{i,i_1},\ldots,\varphi_{i,i_k})$. If $i_{k+1} \ge i_k + 3$ and $i_{k+1} < n$, then $Z(\varphi_{i,i_1},\ldots,\varphi_{i,i_{k+1}})$ consists of $Z(\varphi_{i,i_1},\ldots,\varphi_{i,i_k})$ and three new rows. If $i_{k+1} = n$, $i_{k+1} = i_k + 2$, then $Z(\varphi_{i,i_1},\ldots,\varphi_{i,i_{k+1}})$ consists of $Z(\varphi_{i,i_1},\ldots,\varphi_{i,i_k})$ and a new row. If $i_{k+1} = n$, $i_{k+1} = i_k + s$ and $s \ge 3$, then $Z(\varphi_{i,i_1},\ldots,\varphi_{i,i_{k+1}})$ consists of $Z(\varphi_{i,i_1},\ldots,\varphi_{i,i_k})$ and two new rows. All the latter cases can be similarly proved.
\end{proof}

In the following, first we prove there exists a zero measure set $U_1 \subset \mathbb{R}^{(2nm-m)}$, such that if $K \in \mathbb{R}^{(2nm-m)}\setminus U_1$, the eigenvalues of
$A(K)+D$ are simple. Then we prove that 0 is a regular value of $H(\varphi,\lambda,t)$ for almost all $K \in \mathbb{R}^{(2nm-m)}\setminus (U_1 \cup U_2)$, where $U_2 =(\mathbb{R}^{+})^{(2nm-m)}$. The removal of $U_2$ is to make the elements in the subdiagonal and supdiagonal of the matrix $(1-t)A(K) + D$ negative for $t \in [0,1)$.

\begin{lemma}
\label{sec3:multiple-root}
The eigenvalues of  $A(K)+D$ are simple for $K$ almost everywhere in $\mathbb{R}^{(2nm-m)}$ except on a subset of real codimension 1.
\end{lemma}

\begin{proof}
Let $f(\lambda) = det(A(K) + D -\lambda I)$. The polynomial $f(\lambda)$ has no multiple roots if and only if its discriminant $R(K)$ is nonzero \cite{CT}. It is obvious that $R(K)$ is not identically zero. Furthermore, since $R(K)$ is a polynomial in the elements of vector $K$, it can vanish only on a hypersurface of real codimension 1. The hypersurface is
\begin{equation*}
U_1 = \{K |R(K) = 0\}.
\end{equation*}
\end{proof}

In the following Lemma \ref{sec3:lemma-union-row-full-rank}, we prove that for any matrix $F$ consisting of several submatrices by row, if every submatrix of $F$ is row full rank and the index sets of nonzero columns do not intersect for any two submatrices, then $F$ is row full rank.

\begin{lemma}
\label{sec3:lemma-union-row-full-rank}
Let $F \in \mathbb{R}^{p \times q}$ be a matrix. Suppose $\{ 1,\ldots,p \} = \bigcup_{i=1}^{s}I_i$, with $I_i \bigcap I_j = \emptyset$, if $i \neq j$. Denote
$$
E_i = \{ k\in\{ 1,\ldots,q \}:F(I_i,k) \neq 0 \}.
$$
Suppose that for any $1 \le i \le s$, $F(I_i,:)$ is row full rank and $E_i \bigcap E_j = \emptyset$, if $i \neq j$. Then $F$ is row full rank.
\end{lemma}

\begin{proof}
Since for any $1 \le i \le s$, $F(I_i,:)$ is row full rank, there exist a column index set $J_i \subset E_i$ such that $F(I_i,J_i)$ is a nonsingular submatrix. Correspondingly, the matrix $[F(I_1,:);F(I_2,:);\ldots;F(I_s,:)]$ has a nonsingular submatrix as follows,
\begin{equation}
\left(
  \begin{array}{cccc}
     F(I_1,J_1) &&&\\
    & F(I_2,J_2) &&\\
    &&\ddots&\\
    &&& F(I_s,J_s)
  \end{array}
\right).
\nonumber
\end{equation}
As a result, the matrix $[F(I_1,:);F(I_2,:);\ldots;F(I_s,:)]$ is row full rank and so is $F$.
\end{proof}

To prove that $0$ is a regular value of  $\tilde{H} (\varphi,\lambda,t,K):\mathbb{R}^{mn}\times R \times [0,1)\times  \mathbb{R}^{(2nm-m)}\setminus (U_1\cup U_2) \to \mathbb{R}^{mn+1}$, we need to prove $\forall (\varphi,\lambda,t,K)\in \mathbb{R}^{mn}\times R \times [0,1)\times  \mathbb{R}^{(2nm-m)}\setminus (U_1\cup U_2)$ satisfying $ \tilde{H}(\varphi,\lambda,t,K) = 0$, the Jacobian matrix of $\tilde{H}(\varphi,\lambda,t,K)$ is row full rank.
$\forall (\varphi,\lambda,t,K)$ satisfying $ \tilde{H}(\varphi,\lambda,t,K) = 0$, for $B$ defined in (\ref{sec3:def-B}), we have $B(R^0,:) = 0$. Correspondingly for the Jacobian matrix defined in (\ref{sec3:jacobi-matrix-for-tridiagonal}), we have
\begin{eqnarray*}
((1-t)A + D + 3t\beta \mbox{diag}(\varphi)^2 - \lambda I)(R^0,:) = ((1-t)A + D - \lambda I)(R^0,:).
\end{eqnarray*}
Through row permutations, the Jacobian matrix of $\tilde{H}(\varphi,\lambda,t,K)$, $\frac {\partial \tilde{H}}{\partial(\varphi,\lambda,t,K)}$, can be rewritten as
\begin{equation}
\setlength{\arraycolsep}{2pt}
 \renewcommand{\arraystretch}{0.1}
\left(
\begin{array}{cccc}
    ((1-t)A + D + 3t\beta \mbox{diag}(\varphi)^2 - \lambda I)(R^*,:) & -\varphi(R^*)&(\beta \varphi^3 - A \varphi)(R^*)&(1-t)B(R^*,:)\\
    ((1-t)A + D - \lambda I)(R^0,:) & -\varphi(R^0) & (\beta \varphi^3 - A \varphi)(R^0)&0\\
     -\varphi^\mathrm{T} &0&0&0
\end{array}
\right).
\nonumber
\end{equation}
$\forall (\varphi,\lambda,t,K)$ satisfying $ \tilde{H}(\varphi,\lambda,t,K) = 0$, from Lemma \ref{sec3:lemma-row-full-rank-BR1}, we know $B(R^*,:)$ is row full rank. Therefore, if we can prove
\begin{equation}
\left(\begin{array}{c}
     ((1-t)A + D - \lambda I)(R^0,:)\\
     -\varphi^\mathrm{T}  \end{array}\right)
\end{equation}
is row full rank, then $\frac {\partial \tilde{H}}{\partial(\varphi,\lambda,t,K)}$ is row full rank. From $\tilde{H} = 0$, we have  $((1-t)A + D - \lambda I)(R^0,:)\varphi = 0$, i.e., $\varphi$ is orthogonal to the rows of $((1-t)A + D - \lambda I)(R^0,:)$. Therefore, if we can prove $((1-t)A + D - \lambda I)(R^0,:)$ is row full rank, $\frac {\partial \tilde{H}}{\partial(\varphi,\lambda,t,K)}$ is row full rank. Now the problem is turned into proving that $((1-t)A + D - \lambda I)(R^0,:)$ is row full rank.

For easy exposition, some concepts concerning the topology of the grid points with zero function value are introduced. In addition, $(i,j)$ will be considered as grid point in the rest of this section, representing $(x_i,y_j)$.

\begin{definition}
In the grid $G_a$, a zero valued node $(i,j)$ is a grid point with $\varphi_{i,j}=0$.
\end{definition}

\begin{definition}
Two zero valued nodes $(i_1,j_1)$ and $(i_p,j_p)$ are said to be zero valued connected, if there exist a sequence of zero valued nodes,
\begin{align}
(i_1,j_1),(i_2,j_2),\ldots,(i_p,j_p),
\end{align}
such that for any two successive nodes $(i_k,j_k)$ and $(i_{k+1},j_{k+1})$ of the sequence, the distance of these two nodes is 1 in the sense that
\begin{align}
|i_k-i_{k+1}| + |j_k-j_{k+1}| =1.
\end{align}
\end{definition}

\begin{definition}
A set $S$ consisting of zero valued nodes is called a zero valued connected set if any two nodes of $S$ are zero valued connected.
\end{definition}

\begin{definition}
A set $S$ consisting of zero valued nodes is called a zero valued connected component if $S$ is connected and S is the largest zero connected set containing $S$.
\end{definition}

From (\ref{sec3:eqn-homotopy-3diagonal}), the discretization of the differential equation in a stencil is written explicitly,
\begin{align}
\label{sec3:eqn-in-one-stencil}
\alpha_{1}^{(ij)}\varphi_{ij} + \alpha_{2}^{(ij)}\varphi_{i-1,j}+ \alpha_{3}^{(ij)}\varphi_{i+1,j}+ \alpha_{4}^{(ij)}\varphi_{i,j+1}+\alpha_{5}^{(ij)}\varphi_{i,j-1} =0,\\
\varphi_{0j} = \varphi_{m+1,j} = 0,\\
\varphi_{i0} = \varphi_{i,n+1} = 0,\\
\varphi^{\mathrm{T}}\varphi - \frac{1}{h_1h_2} = 0,
\end{align}
where $\alpha_{1}^{(ij)}=(1-t)a_{jj}^{(i)}+\frac{1}{h_1^2}+\frac{1}{h_2^2} + v_{ij}+ t\beta \varphi_{ij}^2-\lambda$, $\alpha_{2}^{(ij)} = \alpha_{3}^{(ij)} = (1-t)a_{jj+1}^{(i)}-\frac{1}{2h_2^2}$, $\alpha_{4}^{(ij)} = \alpha_{5}^{(ij)} = -\frac{1}{2h_1^2}$, and $i=1,\ldots,m$, $j=1,\ldots,n$. When $K \in R^{(2nm-m)}\setminus (U_1\bigcup U_2)$, the sign relationships among the components of $\varphi$ are given in the following remark.

\begin{remark}
\label{sec3:remark-0relation-in-one-stencil}
Let $(i,j)$ be an inner zero valued node. Assume that all except two of its neighbouring points are known to be zero valued nodes. If one of the rest two points is a zero valued node, then so is the other; If one of the rest two points is not a zero valued node, then neither is the other.
\end{remark}

For any set $G$ of zero valued nodes, denote by $R_G$ the set of indices of rows corresponding to $G$, in which the matrix $B$ is zero, i.e.,
\begin{align}
\label{sec3:def-RG}
R_G= \{ r:B(r,:)=0,r = j + (i-1)*n, (i,j)\in G \}.
\end{align}
Note that if a zero valued node $(i,j) \in G$ is such that $B(s,:)=0$ with $s = j + (i-1)*n$, then the neighbouring inner grid points in $y$-direction should be zero valued nodes, that is, both $(i,j-1)$ and $(i,j+1)$ are zero valued nodes if $1<j<n$, or $(i,j+1)$ is a zero valued node if $j=1$, or $(i,j-1)$ is a zero valued node if $j=n$. If $(i,j) \in G$ and the upper point $(i,j+1)$ or the lower point $(i,j-1)$ is not a zero valued node, then the corresponding row index of $s=\Gamma(i,j)$ will not be in $R_G$. Therefore $\Gamma^{-1}(R_G) \subset G$. Note that $R_G$ may be empty even if $G$ is not empty.

Let $M = (1-t)A + D -\lambda I$. Denote by $C_G$ the indices of columns corresponding to the zero valued connected set $G$, in which $M(R_G,:)$ is not zero, i.e.,
\begin{align}
C_G = \{ c\in R_a: M(R_G,c) \neq 0 \}.
\end{align}
If $R_G = \emptyset$, define $C_G = \emptyset$. For any $s \in R_G$, let $(i,j) = \Gamma^{-1}(s)$. Then the $\Gamma$ images of $(i,j)$ and its neighbouring inner grid points are possibly included in $C_G$.

Denote by $g_i$ the $i$-th column of grid points in $x$-direction, i.e.,
\begin{align}
g_i = \{ (i,j): 1 \leq j \leq n \},
\end{align}
and by $O_1$ ($O_m$) the ordering of all the inner grid points except the first (last) column in $y$-direction,
\begin{align}
O_1 = \{ r=j+(i-1)*n: (i,j) \in G_a\backslash g_1 \}, \\
O_m = \{ r=j+(i-1)*n: (i,j) \in G_a\backslash g_m \}.
\end{align}

\begin{lemma}
\label{sec3:lemma-row-full-rank-O1-Om}
Both $M(O_1,:)$ and $M(O_m,:)$ are row full rank.
\end{lemma}

\begin{proof}
It is obvious that $M(O_1,:)$ has the following form:
\begin{equation}
\left(
  \begin{array}{ccccc}
    *& \ldots &\ldots & \ldots & \ldots \\
    & * & \ldots & \ldots & \ldots\\
    & & \ddots & \ldots & \ldots \\
    & & & * & \ldots
  \end{array}
\right), \nonumber
\end{equation}
where $*$ represents nonzero element. Therefore $M(O_1,:)$ is row full rank.

$M(O_m,:)$ has the following form:
\begin{equation}
\left(
  \begin{array}{ccccc}
    \ldots & * &  &  & \\
    \ldots & \ldots & * &  &  \\
    \ldots & \ldots & \ldots & \ddots &  \\
    \ldots & \ldots & \ldots & \ldots & *
  \end{array}
\right). \nonumber
\end{equation}
Therefore $M(O_m,:)$ is row full rank.
\end{proof}

Note that if $G^0 \bigcap g_1 = \emptyset$, then $R^0 \subset O_1$, from Lemma \ref{sec3:lemma-row-full-rank-O1-Om}, $M(R^0,:)$ is row full rank. In the following, we consider the case $G^0 \bigcap g_1 \not= \emptyset$. The set $G^0 \bigcap g_1$ can have its own zero valued connected components.

\begin{lemma}
\label{sec3:lemma-span-of-zero-valued-nodes}
For $m \ge n \ge 6$, suppose that there is a zero valued connected component $\gamma$ of $G^0 \bigcap g_1$ with $s$ points,  $s \geq 2$. Then
\begin{enumerate}[(i)]
\item  $s < n$;
\item If $\gamma$ contains the point $(1,1)$ or contains the point $(1,n)$, then the zero valued connected component $G$ of $G^0$ containing $\gamma$ will be the set of zero valued nodes starting from $\gamma$ and ending on column $s$ with $s+1-i$ zero valued nodes on column $i$, $1 \leq i \leq s$;
\item Suppose that $\gamma$ is located in the inner part of $g_1$, i.e., the grid point $(1,j)$ of $\gamma$ is such that $2 \leq j \leq n-1$. If $s$ is even, then the zero valued connected component $G$ of $G^0$ containing $\gamma$ will be the set of zero valued nodes starting from $\gamma$ and ending on column $s/2$ with $s+2-2i$ zero valued nodes on column $i$, $1 \leq i \leq s/2$;
\item Suppose that $\gamma$ is located in the inner part of $g_1$. If $s$ is odd, then there is a zero valued connected set $G$ of $G^0$ containing $\gamma$ and arriving at a single zero valued node on column $(s+1)/2$ with $s+2-2i$ zero valued nodes on column $i$, $1 \leq i \leq (s+1)/2$;
\item For the cases (2), (3), (4), if $s \geq 3$, the index set of nonzero columns of $M$ corresponding to $G$ satisfies $C_{G} = \Gamma(G)$.
\end{enumerate}
\end{lemma}

\begin{proof}
\begin{enumerate}[(i)]
\item If $s = n$, we have $\varphi_{11} = \varphi_{12} = \ldots = \varphi_{1n} = 0$. From the sign relationships among the components of $\varphi$ as stated in Remark \ref{sec3:remark-0relation-in-one-stencil}, the grid points in the column 2 are all zero valued nodes, namely, $\varphi_{21} = \varphi_{22} = \ldots = \varphi_{2n} = 0$. By induction, all the grid points are zero valued nodes, namely, $\varphi_{11} = \ldots  = \varphi_{1n} = \ldots = \varphi_{mn} = 0$, a contradiction with the condition that $\varphi^{T}\varphi \neq 0$. Therefore $s < n$.
\item If $\gamma$ starts from the point $(1,1)$, from the sign relationships among the components of $\varphi$, the zero valued nodes connecting $\gamma$ on the column 2 of grid points are $(2,j)$, $1 \leq j \leq s-1$. By induction, it can be seen that the zero valued connected component $G$ of $G^0$ containing $\gamma$ will end on column $s$. The zero valued nodes connecting $\gamma$ are illustrated in Fig \ref{sec3:fig-illustration-flag-1} and form a zero valued component, where black dot represents zero and white dot represents nonzero. Similarly, the case $\gamma$ ends at the point $(1,n)$ can be proved.
\begin{figure}
\caption{Illustration of a flag}
\begin {center}
\label{sec3:fig-illustration-flag-1}
\begin{tikzpicture}[xscale = 0.99,yscale = 0.6]
\draw[thick] (0,0) grid (7,7);
\node at (-0.4,0) {\tiny{$(1,1)$}};
\node at (-0.4,1) {\tiny{$(1,2)$}};
\node at (-0.7,3) {\tiny{$(1,s-1)$}};
\node at (-0.4,4) {\tiny{$(1,s)$}};
\node at (-0.7,5) {\tiny{$(1,s+1)$}};
\node at (-0.4,7) {\tiny{$(1,n)$}};
\draw[dashed] (1,4) to (4,1);
\draw[fill] (0,0) circle [radius = 0.08];
\draw[fill] (0,1) circle [radius = 0.08];
\draw[fill] (0,3) circle [radius = 0.08];
\draw[fill] (0,4) circle [radius = 0.08];
\draw (0,5) circle [radius = 0.1];
\node at (0.7,0.2) {\tiny{$(2,1)$}};
\draw[fill] (1,0) circle [radius = 0.08];
\draw[fill] (1,3) circle [radius = 0.08];
\draw (1,4) circle [radius = 0.1];
\node at (2.5,0.2) {\tiny{$(s-1,1)$}};
\draw[fill] (3,0) circle [radius = 0.08];
\draw[fill] (3,1) circle [radius = 0.08];
\node at (3.7,0.2) {\tiny{$(s,1)$}};
\draw[fill] (4,0) circle [radius = 0.08];
\draw (4,1) circle [radius = 0.08];
\node at (4.5,0.2) {\tiny{$(s+1,1)$}};
\draw (5,0) circle [radius = 0.08];
\node at (6.6,0.2) {\tiny{$(m,1)$}};
\end{tikzpicture}
\end{center}
\end{figure}

\item Assume that $\gamma$ starts from the point $(1,j_1)$ and ends at the point $(1,j_2)$ with $j_1 \geq 2$ and $j_2 \leq n-1$. From the sign relationships among the components of $\varphi$, the zero valued nodes connecting $\gamma$ on the column 2 of grid points are $(2,j)$, $3 \leq j \leq s-2$. The zero valued nodes connecting $\gamma$ are illustrated in Fig \ref{sec3:fig-illustration-flag-2}. The zero valued nodes connecting $\gamma$ will end on column $s/2$ with two zero valued nodes and form a zero valued component $G$.
\begin{figure}
\caption{Illustration of a flag}
\begin {center}
\label{sec3:fig-illustration-flag-2}
\begin{tikzpicture}[xscale = 1.1,yscale = 0.41]
\draw[thick] (0,0) grid (6,13);
\node at (-0.4,0) {\tiny{$(1,1)$}};
\node at (-0.6,1) {\tiny{$(1,t-1)$}};
\node at (-0.4,2) {\tiny{$(1,t)$}};
\node at (-0.6,3) {\tiny{$(1,t+1)$}};
\node at (-0.83,5) {\tiny{$(1,t+\frac{s}{2}-2)$}};
\node at (-0.83,6) {\tiny{$(1,t+\frac{s}{2}-1)$}};
\node at (-0.8,7) {\tiny{$(1,t+\frac{s}{2})$}};
\node at (-0.83,8) {\tiny{$(1,t+\frac{s}{2}+1)$}};
\node at (-0.83,10) {\tiny{$(1,t+s-2)$}};
\node at (-0.8,11) {\tiny{$(1,t+s-1)$}};
\node at (-0.6,12) {\tiny{$(1,t+s)$}};
\node at (-0.4,13) {\tiny{$(1,n)$}};
\draw[dashed] (1,2) to (4,5);
\draw[dashed] (1,3) to (3,5);
\draw[dashed] (1,10) to (3,8);
\draw[dashed] (1,11) to (4,8);
\draw (0,1) circle [radius = 0.12];
\draw[fill] (0,2) circle [radius = 0.08];
\draw[fill] (0,3) circle [radius = 0.08];
\draw[fill] (0,5) circle [radius = 0.08];
\draw[fill] (0,6) circle [radius = 0.08];
\draw[fill] (0,7) circle [radius = 0.08];
\draw[fill] (0,8) circle [radius = 0.08];
\draw[fill] (0,10) circle [radius = 0.08];
\draw[fill] (0,11) circle [radius = 0.08];
\draw (0,12) circle [radius = 0.12];
\node at (0.7,0.3) {\tiny{$(2,1)$}};
\draw(1,2) circle [radius = 0.12];
\draw[fill] (1,3) circle [radius = 0.08];
\draw[fill] (1,5) circle [radius = 0.08];
\draw[fill] (1,6) circle [radius = 0.08];
\draw[fill] (1,7) circle [radius = 0.08];
\draw[fill] (1,8) circle [radius = 0.08];
\draw[fill] (1,10) circle [radius = 0.08];
\draw (1,11) circle [radius = 0.12];
\node at (2.5,0.4) {\tiny{$(\frac{s}{2}-1,1)$}};
\draw[fill] (3,5) circle [radius = 0.08];
\draw[fill] (3,6) circle [radius = 0.08];
\draw[fill] (3,7) circle [radius = 0.08];
\draw[fill] (3,8) circle [radius = 0.08];
\node at (3.7,0.4) {\tiny{$(\frac{s}{2},1)$}};
\draw (4,5) circle [radius = 0.12];
\draw[fill] (4,6) circle [radius = 0.08];
\draw[fill] (4,7) circle [radius = 0.08];
\draw (4,8) circle [radius = 0.12];
\node at (6.4,0.2) {\tiny{$(m,1)$}};
\end{tikzpicture}
\end{center}
\end{figure}

\item Similar to the case (3), the zero valued nodes connecting $\gamma$ are illustrated in Fig \ref{sec3:fig-illustration-flag-3}. However, since $s$ is odd, the zero valued nodes connecting $\gamma$ arrives at one single point on column $(s+1)/2$. All these zero valued nodes connecting $\gamma$ from column 1 to column $(s+1)/2$ form a zero valued connected set $G$. It may be or may not be a zero valued connected component.

\begin{figure}
\caption{Illustration of a flag}
\begin {center}
\label{sec3:fig-illustration-flag-3}
\begin{tikzpicture}[xscale = 1.1,yscale = 0.41]
\draw[thick] (0,0) grid (6,12);
\node at (-0.4,0) {\tiny{$(1,10)$}};
\node at (-0.6,1) {\tiny{$(1,t-1)$}};
\node at (-0.4,2) {\tiny{$(1,t)$}};
\node at (-0.6,3) {\tiny{$(1,t+1)$}};
\node at (-0.8,5) {\tiny{$(1,t+\frac{s-3}{2})$}};
\node at (-0.8,6) {\tiny{$(1,t+\frac{s-1}{2})$}};
\node at (-0.8,7) {\tiny{$(1,t+\frac{s+1}{2})$}};
\node at (-0.8,9) {\tiny{$(1,t+s-2)$}};
\node at (-0.8,10) {\tiny{$(1,t+s-1)$}};
\node at (-0.6,11) {\tiny{$(1,t+s)$}};
\node at (-0.4,12) {\tiny{$(1,n)$}};
\draw[dashed] (1,2) to (4,5);
\draw[dashed] (1,3) to (3,5);
\draw[dashed] (1,9) to (3,7);
\draw[dashed] (1,10) to (4,7);
\draw (0,1) circle [radius = 0.12];
\draw[fill] (0,2) circle [radius = 0.08];
\draw[fill] (0,3) circle [radius = 0.08];
\draw[fill] (0,4) circle [radius = 0.08];
\draw[fill] (0,5) circle [radius = 0.08];
\draw[fill] (0,6) circle [radius = 0.08];
\draw[fill] (0,7) circle [radius = 0.08];
\draw[fill] (0,8) circle [radius = 0.08];
\draw[fill] (0,9) circle [radius = 0.08];
\draw[fill] (0,10) circle [radius = 0.08];
\draw (0,11) circle [radius = 0.12];
\draw (0,5) circle [radius = 0.08];
\node at (0.7,0.3) {\tiny{$(2,1)$}};
\draw(1,2) circle [radius = 0.12];
\draw[fill] (1,3) circle [radius = 0.08];
\draw[fill] (1,4) circle [radius = 0.08];
\draw[fill] (1,5) circle [radius = 0.08];
\draw[fill] (1,6) circle [radius = 0.08];
\draw[fill] (1,7) circle [radius = 0.08];
\draw[fill] (1,8) circle [radius = 0.08];
\draw[fill] (1,9) circle [radius = 0.08];
\draw (1,10) circle [radius = 0.12];
\node at (2.5,0.4) {\tiny{$(\frac{s-1}{2},1)$}};
\draw[fill] (3,5) circle [radius = 0.08];
\draw[fill] (3,6) circle [radius = 0.08];
\draw[fill] (3,7) circle [radius = 0.08];
\node at (3.55,0.4) {\tiny{$(\frac{s+1}{2},1)$}};
\draw (4,5) circle [radius = 0.12];
\draw[fill] (4,6) circle [radius = 0.08];
\draw (4,7) circle [radius = 0.12];
\node at (6.4,0.2) {\tiny{$(m,1)$}};
\end{tikzpicture}
\end{center}
\end{figure}
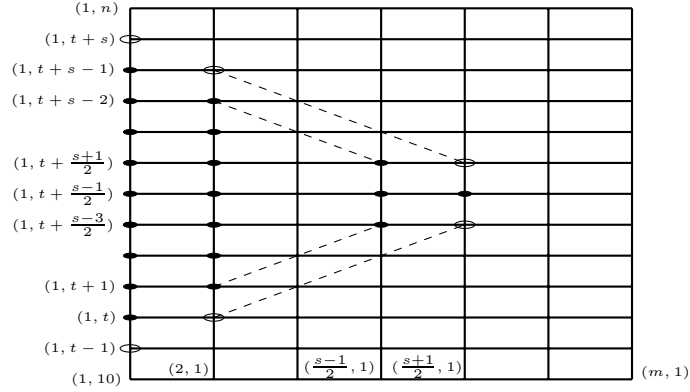

\item From the figures for cases (2), (3), (4), it can be seen that $\Gamma^{-1}(R_{G})$ is the set of zero valued nodes by shrinking $G$ one layer from its nonzero boundary. $\Gamma^{-1}(C_{G})$ is the set of zero valued nodes by extending $\Gamma^{-1}(R_{G})$ one layer towards the nonzero boundary, which is $G$ itself.
\end{enumerate}
\end{proof}

\begin{definition}
Let $\gamma$ be a zero valued connected component of $G^0 \bigcap g_1$ with $s$ points, $s \geq 2$. The zero valued connected set $G$ mentioned in (2), (3), (4) in Lemma \ref{sec3:lemma-span-of-zero-valued-nodes} is called a flag of $\gamma$, denoted as $F_{\gamma}$.
\end{definition}

\begin{lemma}
\label{sec3:lemma-union-row-full-rank-two-zero-sets}
Let $\gamma$ be a zero valued connected component of $G^0 \bigcap g_1$ of $s$ points with $s > 2$ an odd number and $F_{\gamma}$ be its flag. Let $G$ be another set of zero valued nodes. Let $I_1 = R_{F_{\gamma}}$ and $I_2 = R_G$ be the row index sets as defined in (\ref{sec3:def-RG}) such that $M(I_2,:)$ is row full rank.
Denote
\begin{align}
J_1 = \{ j\in R_a: M(I_1,j) \neq 0\}, \quad J_2 = \{ j\in R_a: M(I_2,j) \neq 0\},\\
J_{12} = J_1 \bigcap J_2.
\end{align}
Suppose that $J_{12}$ has only one element $k$ and that both the column $M(I_1,k)$ and the column $M(I_2,k)$ have only one nonzero element, denoted as $M(s_1,k) \neq 0$ and $M(s_2,k) \neq 0$ respectively. Let $(i_1,j_1)$ and $(i_2,j_2)$ be the grid point corresponding to $s_1$ and $s_2$ respectively. Suppose $i_2 = i_1 + 2$ and $j_1 = j_2$.  Then $[M(I_1,:);M(I_2,:)]$ is row full rank.
\end{lemma}

\begin{proof}
Note that since $J_{12}$ is not empty, the zero valued connected component $\gamma$ corresponds to the case (4) of Lemma \ref{sec3:lemma-span-of-zero-valued-nodes}. Geometrically, for that $k$, the two grid points $(i_1,j_1)$ and $(i_2,j_2)$ corresponding to $M(s_1,k) \neq 0$ and $M(s_2,k) \neq 0$ lie in the same row of grid points with distance 2 in the sense that $|i_2 - i_1| + |j_2 - j_1| = 2$, as illustrated locally in Figure \ref{sec3:fig-flag-connection}.
\begin{figure}
\caption{Local illustration of connection with a flag}
\begin {center}
\label{sec3:fig-flag-connection}
\begin{tikzpicture}[xscale = 0.7,yscale = 0.7]
\draw[thick] (0,0) grid (2,2);
\draw[fill] (0,0) circle [radius = 0.1];
\draw[fill] (0,1) circle [radius = 0.1];
\draw[fill] (0,2) circle [radius = 0.1];
\draw[fill] (2,0) circle [radius = 0.1];
\draw[fill] (2,1) circle [radius = 0.1];
\draw[fill] (2,2) circle [radius = 0.1];
\draw[fill] (1,1) circle [radius = 0.1];
\node at (-0.7,1) {\tiny{$(i_1,j_1)$}};
\node at (2.7,1) {\tiny{$(i_2,j_2)$}};
\end{tikzpicture}
\end{center}
\end{figure}
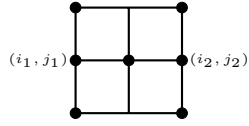
Note that by Lemma \ref{sec3:lemma-row-full-rank-O1-Om}, $M(I_1,:)$ is row full rank.
By row permutations and column permutations, $[M(I_1,:);M(I_2,:)]$ is transformed to the following form, denoted as $W$,

$
$\setlength{\arraycolsep}{2pt}$
$\renewcommand{\arraystretch}{0.1}$
\begin{array}{l@{\hspace{-5pt}}llll}
\begin{array}{@{\hspace{10pt}}l@{\hspace{4pt}}l@{\hspace{2pt}}l@{\hspace{18pt}}l}
\overbrace{\hphantom{\begin{array}{ccccccc}
\ast&\ast&\ast&\ddots&\ddots&\ddots&\ast\end{array}}}^{\displaystyle s}
& \overbrace{\hphantom{\begin{array}{lcccc}
\ast&\ddots&\ddots& \ast&\ast  \end{array}}}^{\displaystyle s-2}
& \overbrace{\hphantom{\begin{array}{lcc}
\ast&\ddots& \ast \end{array}}}^{\displaystyle s-4}
&\overbrace{\hphantom{\begin{array}{lcc}
\ast & \ast&\ast  \end{array}}}^{\displaystyle 3}

\end{array}
&\\
\left(\begin{array}{ccccccc|ccccc|ccc|c|ccc|c|cc}

    \ast&\ast&\ast&&&&&\ast&&&&&&&&&&&&&&\\
    &\ast&\ast&\ast&&&&&\ast&&&&&&&&&&&&&\\
    &&\ast&\ast&\ast&&&&&\ast&&&&&&&&&&&&\\
    &&&\ddots&\ddots&\ddots&&&&&\ddots&&&&&&&&&&&\\
    &&&&\ast&\ast&\ast&&&&&\ast&&&&&&&&&\\
    &&\ast&&&&&\ast&\ast&\ast&&&\ast&&&&&&&&&\\
    &&&\ddots&&&&&\ddots&\ddots&\ddots&&&\ddots&&&&&&&&\\
    &&&&\ast&&&&&&\ast&\ast&&&\ast&&&&&&&\\
    &&&&&&&&&\ddots&&&&\ddots&&\ddots&&&&&&\\
    &&&&&&&&&&&&&&&&\ast&&&&&\\
    &&&&&&&&&&&&&&&&&\ast&&&&\\
    &&&&&&&&&&&&&&&&&&\ast&&&\\
    &&&&&&&&&&&&&&&\cdots&\ast&\ast&\ast&\ast&&\\ \hline
   &&&&&&&&&&&&&&&&&&&\ast&\ast&\cdots\\
 &&&&&&&&&&&&&&&&&&&&\ast&\\
 &&&&&&&&&&&&&&&&&&&&&\ddots
 \end{array}\right) &
\begin{array}{l}
\vspace{35pt}
\left.\vphantom{\begin{array}{c}
  \ast\\
\ast\\
\ast\\
\ddots\\
\ast\\
\ast\\
\ddots\\
\ast\\
\ddots\\
\ast\\
\ast\\
\ast\\
\ast \end{array}}
\right\}\frac{(s-1)^2}{4}
\end{array}
\end{array}.
$

Note that the diagonal entries of the submatrix $W(1:\frac{(s-1)^2}{4},s+1:\frac{(s+1)^2}{4})$ are all nonzero. Therefore,
by elementary matrix transformations on the first $\frac{(s+1)^2}{4}$ columns, the above matrix $W$ is transformed to the following, denoted as $\overline{W}$,

$
$\setlength{\arraycolsep}{2pt}$
$\renewcommand{\arraystretch}{0.1}$
\begin{array}{l@{\hspace{-5pt}}llll}
\begin{array}{@{\hspace{1pt}}l@{\hspace{2pt}}l@{\hspace{2pt}}l@{\hspace{15pt}}l}
\overbrace{\hphantom{\begin{array}{ccccccc}
\ast&\ast&\ast&\ddots&\ddots&\ddots&\ast\end{array}}}^{\displaystyle s}
& \overbrace{\hphantom{\begin{array}{lcccc}
\ast&\ddots&\ddots& \ast&\ast  \end{array}}}^{\displaystyle s-2}
& \overbrace{\hphantom{\begin{array}{lcc}
\ast&\ddots& \ast \end{array}}}^{\displaystyle s-4}
&\overbrace{\hphantom{\begin{array}{lcc}
\ast & \ast&\ast  \end{array}}}^{\displaystyle 3}

\end{array}
&\\
\left(\begin{array}{ccccccc|ccccc|ccc|c|ccc|c|cc}

    &&&&&&&\ast&&&&&&&&&&&&&&\\
    &&&&&&&&\ast&&&&&&&&&&&&&\\
    &&&&&&&&&\ast&&&&&&&&&&&&\\
    &&&&&&&&&&\ddots&&&&&&&&&&&\\
    &&&&&&&&&&&\ast&&&&&&&&&\\
    &&&&&&&&&&&&\ast&&&&&&&&&\\
    &&&&&&&&&&&&&\ddots&&&&&&&&\\
    &&&&&&&&&&&&&&\ast&&&&&&&\\
    &&&&&&&&&&&&&&&\ddots&&&&&&\\
    &&&&&&&&&&&&&&&&\ast&&&&&\\
    &&&&&&&&&&&&&&&&&\ast&&&&\\
    &&&&&&&&&&&&&&&&&&\ast&&&\\
    &&&&&&&&&&&&&&&&&&&\ast&&\\ \hline
   \times&\ast&\ast&\cdots&\ast&\ast&\ast&\ast&\ast&\cdots&\ast&\ast&\ast&\cdots&\ast&\cdots&\ast&\ast&\ast&\ast&\ast&\cdots\\
 &&&&&&&&&&&&&&&&&&&&\ast&\\
 &&&&&&&&&&&&&&&&&&&&&\ddots
 \end{array}\right) &
\begin{array}{l}
\vspace{35pt}
\left.\vphantom{\begin{array}{c}
  \ast\\
\ast\\
\ast\\
\ddots\\
\ast\\
\ast\\
\ddots\\
\ast\\
\ddots\\
\ast\\
\ast\\
\ast\\
\ast \end{array}}
\right\}\frac{(s-1)^2}{4}
\end{array}
\end{array}.$

By elementary matrix transformations on the first $\frac{(s-1)^2}{4}+1$ rows, the above matrix $\overline{W}$ is further transformed to the following, denoted as $\overline{\overline{W}}$,

$
$\setlength{\arraycolsep}{2.3pt}$
$\renewcommand{\arraystretch}{0.1}$
\begin{array}{l@{\hspace{-5pt}}llll}
\begin{array}{@{\hspace{0.1pt}}l@{\hspace{1pt}}l@{\hspace{1pt}}l@{\hspace{9pt}}l}
\overbrace{\hphantom{\begin{array}{ccccccc}
\ast&\ast&\ast&\ddots&\ddots&\ddots&\ast\end{array}}}^{\displaystyle s}
& \overbrace{\hphantom{\begin{array}{lcccc}
\ast&\ddots&\ddots& \ast&\ast  \end{array}}}^{\displaystyle s-2}
& \overbrace{\hphantom{\begin{array}{lcc}
\ast&\ddots& \ast \end{array}}}^{\displaystyle s-4}
&\overbrace{\hphantom{\begin{array}{lcc}
\ast & \ast&\ast  \end{array}}}^{\displaystyle 3}

\end{array}
&\\
\left(\begin{array}{ccccccc|ccccc|ccc|c|ccc|c|cc}

    &&&&&&&\ast&&&&&&&&&&&&&&\\
    &&&&&&&&\ast&&&&&&&&&&&&&\\
    &&&&&&&&&\ast&&&&&&&&&&&&\\
    &&&&&&&&&&\ddots&&&&&&&&&&&\\
    &&&&&&&&&&&\ast&&&&&&&&&\\
    &&&&&&&&&&&&\ast&&&&&&&&&\\
    &&&&&&&&&&&&&\ddots&&&&&&&&\\
    &&&&&&&&&&&&&&\ast&&&&&&&\\
    &&&&&&&&&&&&&&&\ddots&&&&&&\\
    &&&&&&&&&&&&&&&&\ast&&&&&\\
    &&&&&&&&&&&&&&&&&\ast&&&&\\
    &&&&&&&&&&&&&&&&&&\ast&&&\\
    &&&&&&&&&&&&&&&&&&&\ast&&\\ \hline
   \times&\ast&\ast&\cdots&\ast&\ast&\ast&&&&&&&&&&&&&&\ast&\cdots\\
 &&&&&&&&&&&&&&&&&&&&\ast&\\
 &&&&&&&&&&&&&&&&&&&&&\ddots
 \end{array}\right) &
\begin{array}{l}
\vspace{35pt}
\left.\vphantom{\begin{array}{c}
  \ast\\
\ast\\
\ast\\
\ddots\\
\ast\\
\ast\\
\ddots\\
\ast\\
\ddots\\
\ast\\
\ast\\
\ast\\
\ast \end{array}}
\right\}\frac{(s-1)^2}{4}
\end{array}
\end{array}.$

Note that in the row $\frac{(s-1)^2}{4}+1$, at least one element $\overline{\overline{W}}(\frac{(s-1)^2}{4}+1,1)$, denoted as $'\times'$, is not zero. Therefore, the lower submatrix $\overline{\overline{W}}(\frac{(s-1)^2}{4}+1:end,:)$ is still row full rank. Now Lemma \ref{sec3:lemma-union-row-full-rank} can be applied to conclude that $\overline{\overline{W}}$ is row full rank. Therefore $[M(I_1,:);M(I_2,:)]$ is row full rank.
\end{proof}

\begin{theorem}
\label{sec3:thm-0-is-regular-value}
For the homotopy $H:\mathbb{R}^{mn}\times \mathbb{R}\times [0,1) \rightarrow \mathbb{R}^{mn}\times \mathbb{R}$ in (\ref{sec3:eqn-homotopy-3diagonal}), $\forall n\in \mathbb{N}^{+}$, $m \ge n \ge 6$, for almost all $K \in  \mathbb{R}^{(2nm-m)}\setminus (U_1\cup U_2)$,
\begin{enumerate}[(i)]
\item 0 is a regular value of $H$ defined in $(\ref{sec3:eqn-homotopy-3diagonal})$ and therefore the homotopy paths corresponding to different initial points do not intersect each other for $t\in[0,1)$;
\item Every homotopy path $(\varphi(s),\lambda(s),t(s)) \subset H^{-1}(0)$ is bounded.
\end{enumerate}
\end{theorem}

\begin{proof}
(i). It suffices to prove that $\forall (\varphi,\lambda,t,K)\in \mathbb{R}^{mn}\times \mathbb{R} \times [0,1)\times  \mathbb{R}^{(2nm-m)}\setminus (U_1\cup U_2)$ satisfying $ \tilde{H} (\varphi,\lambda,t,K) = 0$, $((1-t)A + D - \lambda I)(R^0,:)$ is row full rank, or in short hand notation $M(R^0,:)$ is row full rank. Note that $R^0$ corresponds to the zero valued nodes $G^0$. If $G^0 \bigcap g_1 = \emptyset$, then $R^0 \subset O_1$. By Lemma \ref{sec3:lemma-row-full-rank-O1-Om} $M(R^0,:)$ is row full rank. Next $G^0 \bigcap g_1 \neq \emptyset$ is assumed. If the zero valued connected components of the set $G^0 \bigcap g_1$ are all single point sets, $R^0$ is a subset of $O_1$, and again by Lemma \ref {sec3:lemma-row-full-rank-O1-Om}, $M(R^0,:)$ is row full rank.

Assume that there are $q$ zero valued connected components of the set $G^0 \bigcap g_1$, each of which has more than one point. These connected components of the set $G^0 \bigcap g_1$ are denoted as $\gamma_i$, with corresponding flags $G_i = F_{\gamma_i}$, $i=1,\cdots,q$. Denote
\begin{align}
\overline{G} = G_a \setminus (\bigcup_{i=1}^{q}G_i),\\
\overline{R} = R_{\overline{G}}, \quad \overline{C} = C_{\overline{G}}, \quad R_i = R_{G_i}, \quad C_i = C_{G_i},\quad\quad\forall 1 \leq i \leq q.
\end{align}
By Lemma \ref{sec3:lemma-span-of-zero-valued-nodes}, it can be verified that
\begin{align}
R^0 = \overline{R} \bigcup \left( \bigcup_{i=1}^{q} R_{i} \right),\\
R_{i} \bigcap \overline{R} = \emptyset, \quad\forall 1 \leq i \leq q,\\
R_{i} \bigcap R_{j} =\emptyset, \quad C_{i} \bigcap C_{j} = \emptyset, \quad\forall i \neq j.
\end{align}
Note that if a zero valued connected component $\gamma$ of $G^0 \bigcap g_1$ has only two nodes, then the flag of $\gamma$ is the $\gamma$ itself, and $R_{\gamma} = \emptyset$. Assume that all the zero valued connected components $\gamma_i$ have more than two nodes. Since
\begin{align}
\overline{R} \subset O_1, \quad R_{i} \subset O_m, \quad \forall 1 \leq i \leq q,
\end{align}
by Lemma \ref{sec3:lemma-row-full-rank-O1-Om}, all of $M(\overline{R},:)$ and $M(R_{i},:)$, $1 \leq i \leq q$, are row full rank. It is possible that for some  $\gamma_i$, $C_{i} \bigcap \overline{C} \neq \emptyset$. If so, $\gamma_i$ is the case described in case $(4)$ in Lemma \ref{sec3:lemma-span-of-zero-valued-nodes}, $C_{i} \bigcap \overline{C}$ has only one element and the corresponding grid points are illustrated in Figure \ref{sec3:fig-local-illustration-of-a-connection}.
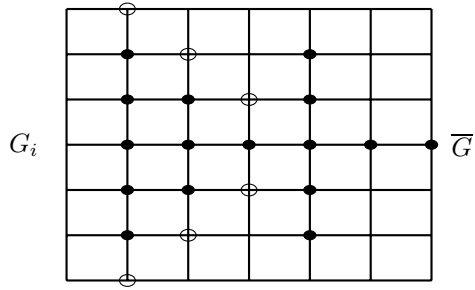
\begin{figure}
\caption{Local illustration of a connection}
\begin {center}
\label{sec3:fig-local-illustration-of-a-connection}
\begin{tikzpicture}[xscale = 0.8,yscale = 0.6]
\draw[thick] (0,0) grid (6,6);
\draw (1,0) circle [radius = 0.13];
\draw[fill] (1,1) circle [radius = 0.1];
\draw[fill] (1,2) circle [radius = 0.1];
\draw[fill] (1,3) circle [radius = 0.1];
\draw[fill] (1,4) circle [radius = 0.1];
\draw[fill] (1,5) circle [radius = 0.1];
\draw (1,6) circle [radius = 0.13];
\draw (2,1) circle [radius = 0.13];
\draw[fill] (2,2) circle [radius = 0.1];
\draw[fill] (2,3) circle [radius = 0.1];
\draw[fill] (2,4) circle [radius = 0.1];
\draw (2,5) circle [radius = 0.13];
\draw (3,2) circle [radius = 0.13];
\draw[fill] (3,3) circle [radius = 0.1];
\draw (3,4) circle [radius = 0.13];
\draw[fill] (4,1) circle [radius = 0.1];
\draw[fill] (4,2) circle [radius = 0.1];
\draw[fill] (4,3) circle [radius = 0.1];
\draw[fill] (4,4) circle [radius = 0.1];
\draw[fill] (4,5) circle [radius = 0.1];
\draw[fill] (5,3) circle [radius = 0.1];
\draw[fill] (6,3) circle [radius = 0.1];
\node at (-0.7,3) {$G_i$};
\node at (6.5,3) {{$\overline{G}$}};
\end{tikzpicture}
\end{center}
\end{figure}
Without loss of generality, suppose that all the $\gamma_i$ satisfying such property are the first $p$ $\gamma_i$, i.e.,
\begin{align}
C_{i} \bigcap \overline{C} \neq \emptyset,\quad 1 \leq i \leq p,\\
C_{i} \bigcap \overline{C} = \emptyset,\quad p+1 \leq i \leq q.
\end{align}
Denote
\begin{align}
G & = \overline{G} \bigcup \left( \bigcup_{i=p+1}^{q}G_i \right),\\
Z(G_k,\ldots,G_1,G) & = [M(R_k,:);\ldots;M(R_1,:);M(R_G,:)],\quad \forall 1\le k \le p.
\end{align}
Then $M(R^0,:) = Z(G_p,\ldots,G_1,G)$. The claim that $M(R^0,:)$ is row full rank will be proved by recursion.

Firstly, prove $Z(G_1,G)$ is row full rank. Take $I_1 = R_{1}$ and $I_2 = R_{G}$. The conditions of Lemma \ref{sec3:lemma-union-row-full-rank-two-zero-sets} are satisfied. Thus $[M(R_1,:);M(R_G,:)]$ is row full rank.

Secondly, prove $Z(G_{k+1},\ldots,G_1,G)$ is row full rank if $Z(G_k,\ldots,G_1,G)$ is, $\forall 1\le k < p$. Take $I_1 = R_{k+1}$ and $I_2 = R_{k} \cup \cdots \cup R_{1} \cup R_{G}$. Since $C_{k+1} \cap C_{i} = \emptyset$, $1 \leq i \leq k$, and $C_{k+1} \bigcap C_{G}$ has only one element, $I_1$ and $I_2$ satisfy the conditions of Lemma \ref{sec3:lemma-union-row-full-rank-two-zero-sets}. Thus $[M(R_k,:);\ldots;M(R_1,:);M(R_G,:)]$ is row full rank.

(ii). Similar to the one dimensional case, that $H^{-1}(0)$ is bounded can be proved also.
\end{proof}

\subsubsection{The homotopy with random pentadiagonal matrix}

A homotopy with random pentadiagonal matrix is also possible. Specifically, for $\tilde{H}$ defined in (\ref{sec3: mapping-related-to-homotopy-2D}), replacing $A(K)$ with $\overline{A}(\overline{K})$, where $\overline{A}(\overline{K}) \in \mathbb{R}^{mn \times mn}$ is a random pentadiagonal matrix with the same sparse structure as $D$, namely,
\begin{equation*}
\overline{A}(\overline{K})=
\setlength{\arraycolsep}{0.5pt}
 \renewcommand{\arraystretch}{0.1}
\left(\begin{array}{cccc}
    A_{1}&\overline{A}_1&&\\
    \overline{A}_1&A_{2}&\overline{A}_2&\\
    &&\ddots&\overline{A}_{m-1}\\
    &&\overline{A}_{m-1}&A_{m}
    \end{array}\right),
\end{equation*}
where
\begin{equation*}
A_{i}=
\left(
\begin{array}{ccccc}
a_{11}^{(i)} & a_{12}^{(i)}&&&\\
a_{12}^{(i)} & a_{22}^{(i)} & a_{23}^{(i)} &&\\
& a_{23}^{(i)} & a_{33}^{(i)} &\ddots&\\
&& \ddots & \ddots & a_{n-1,n}^{(i)}\\
&&& a_{n-1,n}^{(i)} & a_{nn}^{(i)}
\end{array}\right),\quad
\overline{A}_{i}=
\left(
\begin{array}{ccccc}
b_{1}^{(i)} & &&&\\
 & b_{2}^{(i)} & &&\\
& & b_{3}^{(i)} & &\\
&& & \ddots & \\
&&& & b_{n}^{(i)}
\end{array}\right),
\end{equation*}

\begin{align}
\nonumber
\overline{K} =\left( K_1,K_2,\ldots,K_m,\overline{K}_1,\ldots,\overline{K}_{m-1} \right)^{\mathrm{T}},\\
\nonumber
K_i =\left( a_{11}^{(i)},a_{12}^{(i)},a_{22}^{(i)},a_{23}^{(i)},\ldots,a_{n-1,n-1}^{(i)},a_{n-1,n}^{(i)},a_{nn}^{(i)} \right),~ i=1,\ldots,m,\\
\nonumber
\overline{K}_i =\left( b_{1}^{(i)},b_{2}^{(i)},\ldots,b_{n}^{(i)} \right), ~i=1,\ldots,m-1.
\end{align}

The Jacobian matrix of $\tilde{H}(\varphi,\lambda,t,\overline{K})$, $\frac{\partial \tilde{H}}{\partial(\varphi,\lambda,t,\overline{K})}$, is :
\begin{equation}
\label{sec3:jacobi-matrix-for-pentadiagonal}
\nonumber
\setlength{\arraycolsep}{8pt}
\renewcommand{\arraystretch}{0.45}
\left(
\begin{array}{cccc}
(1-t)\overline{A}(\overline{K})+D+3t\beta \mbox{diag}(\varphi^2)-\lambda I &-\varphi& \beta\varphi^3-\overline{A}(\overline{K})\varphi&(1-t)\overline{B}\\
                                              -\varphi^\mathrm{T}&0&0&0
\end{array}
\right),
\end{equation}
where $\overline{B}=\frac{\partial (\overline{A}(\overline{K})\varphi)}{\partial \overline{K}} \in \mathbb{R}^{{(mn)}\times {(3nm-m-n)}}$. Recall $\varphi_i = \left( \varphi_{i1},\varphi_{i2},\ldots,\varphi_{in} \right)^{\mathrm{T}}$. It can be verified that
\begin{align}
\nonumber
\overline{B} & = \left(
  \begin{array}{ccccc|cccc}
     B_{1}&&&&&\overline{B}_2&&&\\
    &B_{2}&&&&\overline{B}_1&\overline{B}_3&&\\
    &&\ddots&&&&\overline{B}_2&\overline{B}_4&\\
    &&&\ddots&&&&\ddots&\ddots\\
    &&&&B_{m}&&&&\overline{B}_{m-1}
  \end{array}
\right) \\
& = \left(
  \begin{array}{ccccc|cccc}
     &&&&&\overline{B}_2&&&\\
    &&&&&\overline{B}_1&\overline{B}_3&&\\
    &&B&&&&\overline{B}_2&\overline{B}_4&\\
    &&&&&&&\ddots&\ddots\\
    &&&&&&&&\overline{B}_{m-1}
  \end{array}
\right),
\label{sec3:def-B-bar}
\end{align}
with
\begin{align}
\overline{B}_i =
\left(
  \begin{array}{cccc}
     \varphi_{i1} & & &\\
    & \varphi_{i2} & &\\
    & & \ddots &\\
    & & & \varphi_{in}
  \end{array}
\right).
\end{align}
Note that the left part $B$ of $\overline{B}$ is nothing but the matrix in (\ref{sec3:def-B}).

\begin{lemma}
The eigenvalues of  $\overline{A}(\overline{K})+D$ are simple for $\overline{K}$ almost everywhere except on a subset of real codimension 1.
\end{lemma}

\begin{proof}
Similar to the proof of Lemma \ref{sec3:multiple-root}.
\end{proof}

For our discussion, $\overline{U}_2 =(\mathbb{R}^{+})^{(3nm-m-n)}$ is removed. To prove that $\forall (\varphi,\lambda,t,\overline{K})\in \mathbb{R}^{mn}\times R \times [0,1)\times  \mathbb{R}^{(3nm-m-n)}\setminus (\overline{U}_1\cup \overline{U}_2)$ satisfying $ \tilde{H}(\varphi,\lambda,t,\overline{K}) = 0$, the Jacobian matrix of $\tilde{H}(\varphi,\lambda,t,\overline{K})$ in (\ref{sec3:jacobi-matrix-for-pentadiagonal}) is row full rank, it suffices to prove that the following submatrix of the Jacobian matrix in (\ref{sec3:jacobi-matrix-for-pentadiagonal})
\begin{equation}
\label{sec3:subjacobian-for-pentadiagonal}
\setlength{\arraycolsep}{8pt}
\renewcommand{\arraystretch}{0.45}
\left(
\begin{array}{cccc}
(1-t)\overline{A}(\overline{K})+D+3t\beta \mbox{diag}(\varphi^2)-\lambda I &-\varphi& \beta\varphi^3-\overline{A}(\overline{K})\varphi&(1-t)B\\
                                              -\varphi^\mathrm{T} &0&0&0
\end{array}
\right)
\end{equation}
is row full rank. With the notations $R^0$ and $R^*$ as in (\ref{sec3:notation-R0-R*}) and similar arguments as in Subsection \ref{sec3:subsubset-tridiagonal-case}, it can be proved that $((1-t)\overline{A}+D-\lambda I))(R^0,:)$ is row full rank. Therefore the matrix (\ref{sec3:subjacobian-for-pentadiagonal}) is row full rank. That is $0$ is a regular value of the homotopy $H$ with random pentadiagonal matrix for almost all $\overline{K} \in \mathbb{R}^{(3nm-m-n)}\setminus (U_1\cup U_2)$ with $m \ge n \ge 6$.

\section{Algorithm and numerical results}

\subsection{Algorithm}
Thanks to Theorems \ref{sec 3: thm-regular one dimen} and \ref{sec3:thm-0-is-regular-value}, since 0 is a regular value of the homotopies constructed, the homotopy paths determined by the homotopy equations (\ref{sec 3: homotopy one dimen}) and (\ref{sec3:eqn-homotopy-3diagonal}) have no bifurcation points with probability one. Therefore the usual path following algorithm, i.e., the predictor-corrector method as in \cite{Allgower,Huang-Zeng-Ma}, can be adapted to trace the homotopy paths of the homotopy equations (\ref{sec 3: homotopy one dimen}) and (\ref{sec3:eqn-homotopy-3diagonal}). The adapted algorithm is stated in Algorithm \ref{algo1:predictor-corrector-homotopy-method}. For notation convenience, in Algorithm \ref{algo1:predictor-corrector-homotopy-method}, we denote the iterate at step $k$ as $x_k = (\varphi^{(k)},\lambda_k)$, $k=1,\ldots$.

\begin{algorithm}
\caption{Predictor-Corrector method}\label{algo1:predictor-corrector-homotopy-method}
\textbf{Initialization:} Set $(x_0,t_0)=(x_0,0)$, $k = 0$, the minimum step size $ds_{min}$, the initial step size $ds$. Compute the tangent vector $(\dot{x}_0,\dot{t}_0)$ such that $\dot{t}_0 > 0$ and record the orientation $ori$
\begin{equation}
{H_x \dot{x}_0 + H_t \dot{t}_0 = 0},\qquad
ori = \mbox{sign}
\left(
\left|
\begin{array}{cc}
H_x & H_t\\
\dot{x}_0 & \dot{t}_0
\end{array}
\right|
\right).
\end{equation}
\While{$t_k<1$}{

\textbf{Predictor:} $( \bar{x}_{k+1}, \bar{t}_{k+1}) = (x_k,t_k) + ds (\bar{x}_{k},\bar{t}_{k})$\;
\If {$\bar{t}_{k+1} > 1$} {change ds such that $\bar{t}_{k+1} = 1$\;}

\textbf{Corrector:}
\uIf {$\bar{t}_{k+1} = 1$} {$(v,\tau) = (0,1)$\;}
\Else {$(v,\tau) = (\dot{x}_{k},\dot{t}_{k})$\;}
Employ Newton method to solve the following nonlinear equations:
\begin{equation}
\begin{array}{ccc}
\left(\begin{array}{c}
H(x,t)\\
v^T(x-\bar{x}_{k+1})+\tau(t-\bar{t}_{k+1})\end{array}\right) & = & 0.
\end{array}
\end{equation}

\textbf{Judgement:}
\uIf {the above iteration converges to $(x_{k+1},t_{k+1})$}{compute the tangent vector $(\dot{x}_{k+1},\dot{t}_{k+1})$ satisfying
\begin{equation}
{H_x \dot{x}_{k+1} + H_t \dot{t}_{k+1} = 0}, \qquad
\mbox{sign}
\left(
\left|
\begin{array}{cc}
H_x & H_t\\
\dot{x}_{k+1} & \dot{t}_{k+1}
\end{array}
\right|
\right) = ori
\end{equation}}
\Else {$ds = \frac{1}{2}ds $\; go to Predictor\;}
Compute angle $\theta$ of $(\bar{x}_{k},\bar{t}_{k})$ and $(\bar{x}_{k+1},\bar{t}_{k+1})$\;
\If {$\theta > 18^0$} {$ds = \frac{1}{2}ds,$\: go to Predictor\;}

\textbf{Accept iterates:} $(x_{k},t_{k}) =(x_{k+1},t_{k+1})$,$(\dot{x}_{k},\dot{t}_{k}) = (\dot{x}_{k+1},\dot{t}_{k+1})$\;

\If {$\theta < 6^0$}{$ds = 2ds$\;}
\If {$ds < ds_{min}$} {stop the algorithm\;}}

\end{algorithm}

\subsection{Numerical results}
The first numerical example is the 1D discretized problem (\ref{sec 2: finite difference - one dimen}) with $\beta = 20$, $V(x)=\frac{1}{2}x^2$, $\Omega=[-2,2]$ and $n=999$. Some eigenvectors, i.e., approximate eigenfunctions are plotted in Figures \ref{sec4:fig-1D-1} to \ref{sec4:fig-1D-9}. The eigenvector in Figure \ref{sec4:fig-1D-1} is the unique positive solution of (\ref{sec 2: finite difference - one dimen}) as stated in Theorem \ref{sec 2: thm-positive-eigenvector}, which corresponds to the unique positive ground state, as proved in \cite{B} for the continuous nonlinear eigenvalue problem (\ref{sec 1: transformed nonlinear problem})-(\ref{sec 1:transformed constraint}) when $\beta>0$. The approximate eigenfunction in Figure \ref{sec4:fig-1D-2} is antisymmetric as described in Theorem \ref{sec 2: certain solution} and is an approximate first excited state. Others are approximate excited states corresponding to higher energy. The order preserving property of the eigenvalue curves as stated in \cite{TYA} was observed, that is, if $\lambda(0)$ is the $k$th smallest eigenvalue of the initial problem, then $\lambda(t)$ is the $k$th smallest eigenvalue of the intermediate problem for each $t \in [0,1)$. However, we are not able to prove such property for the eigenvector-dependent nonlinear eigen-problem yet.

\begin{figure}[H]
     \centering
     \subfigure
     {
         \begin{minipage}[b]{.3\linewidth}
          \centering
          \includegraphics[scale = 0.18]{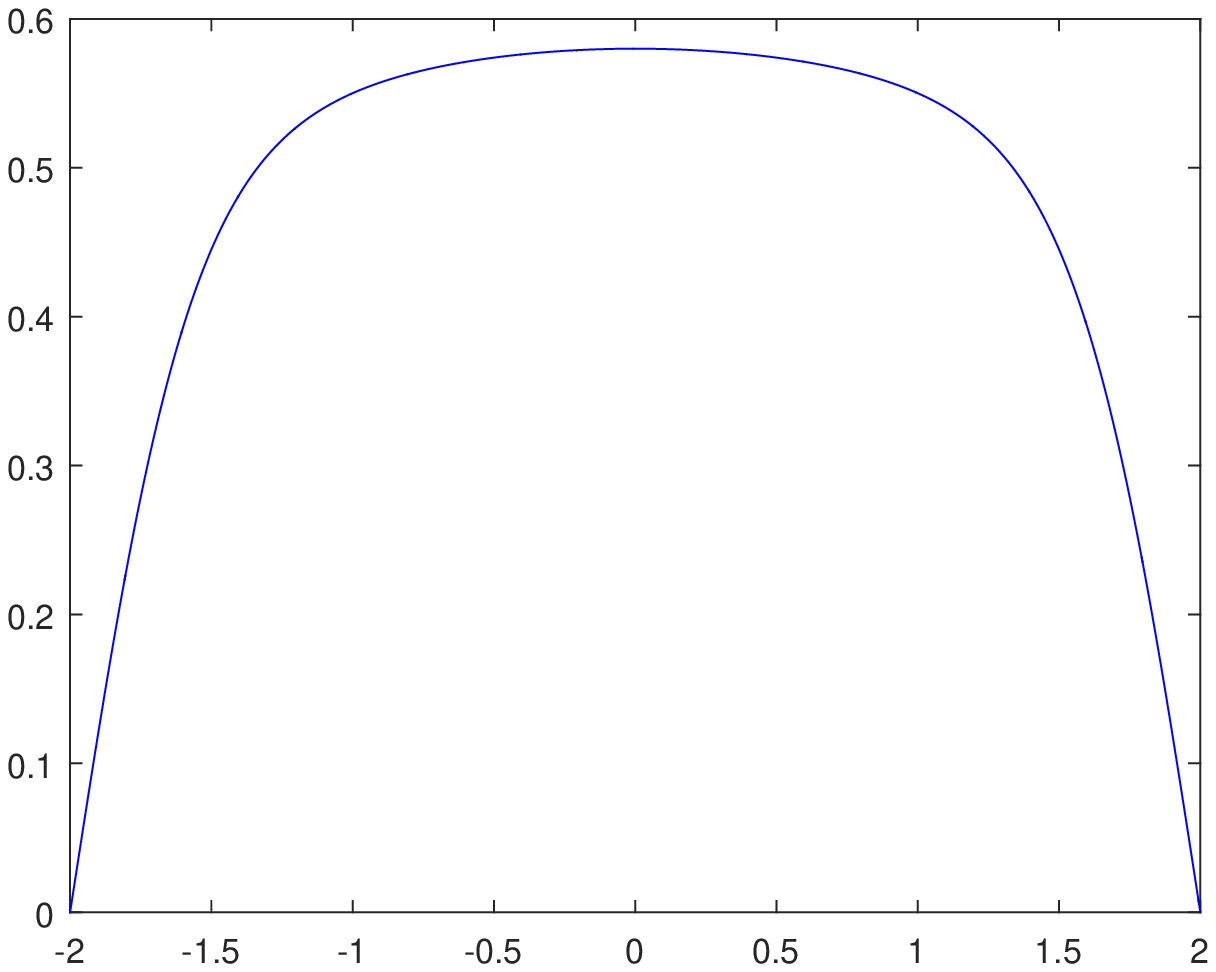}
          \caption{$\lambda = 6.76$}\label{sec4:fig-1D-1}
           \end{minipage}
      }
      \subfigure
      {
          \begin{minipage}[b]{.3\linewidth}
          \centering
          \includegraphics[scale = 0.18]{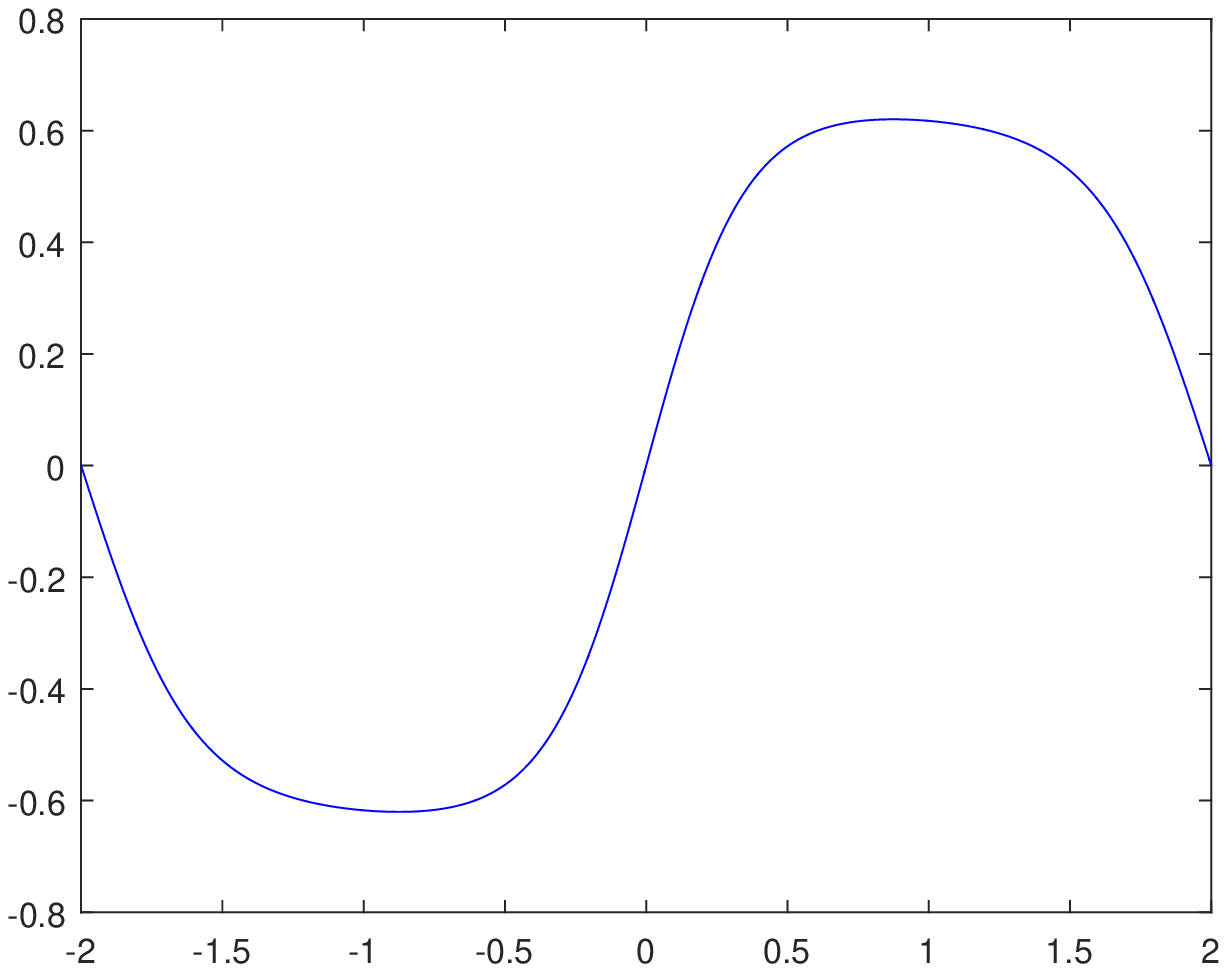}
          \caption{$\lambda = 8.39$}\label{sec4:fig-1D-2}
          \end{minipage}
       }
       \subfigure
      {
          \begin{minipage}[b]{.3\linewidth}
          \centering
          \includegraphics[scale = 0.18]{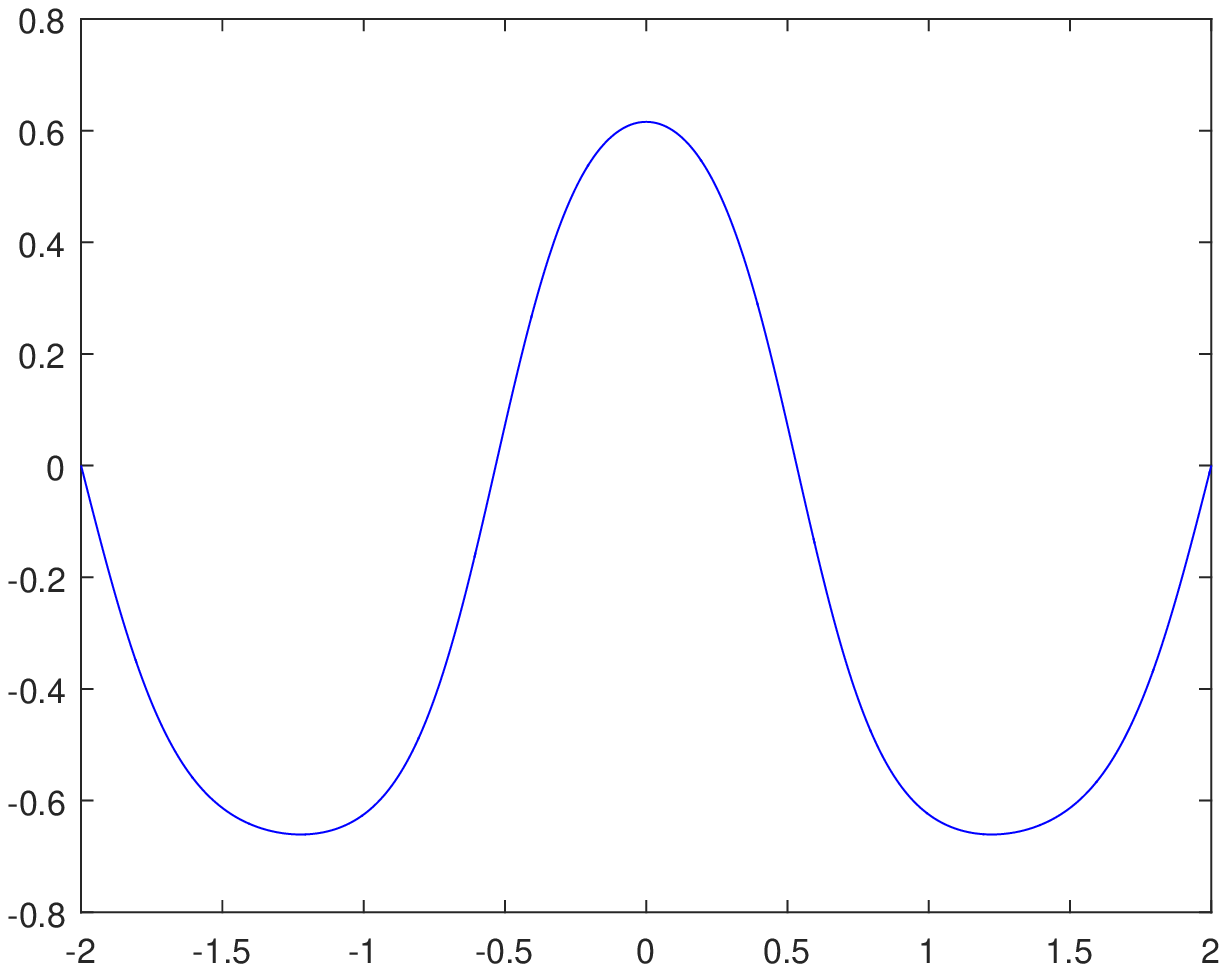}
          \caption{$\lambda = 10.35$}
          \end{minipage}
       }
\end{figure}
\vspace{-1cm}
\begin{figure}[htb]
     \centering
     \subfigure
     {
        \begin{minipage}[b]{.3\linewidth}
        \centering
        \includegraphics[scale = 0.18]{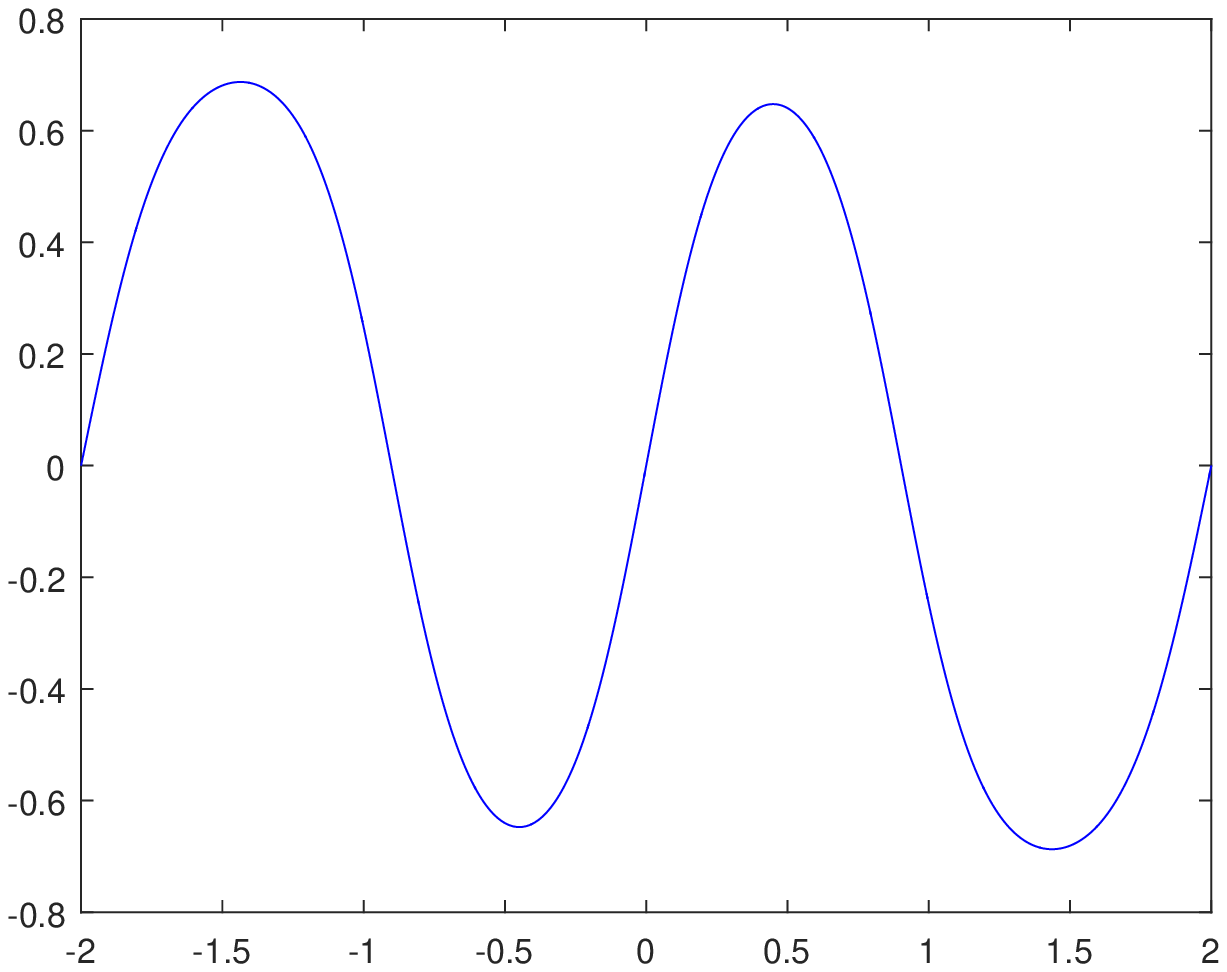}
         \caption{$\lambda = 12.73$}
         \end{minipage}
      }
      \subfigure
      {
         \begin{minipage}[b]{.3\linewidth}
          \centering
          \includegraphics[scale = 0.18]{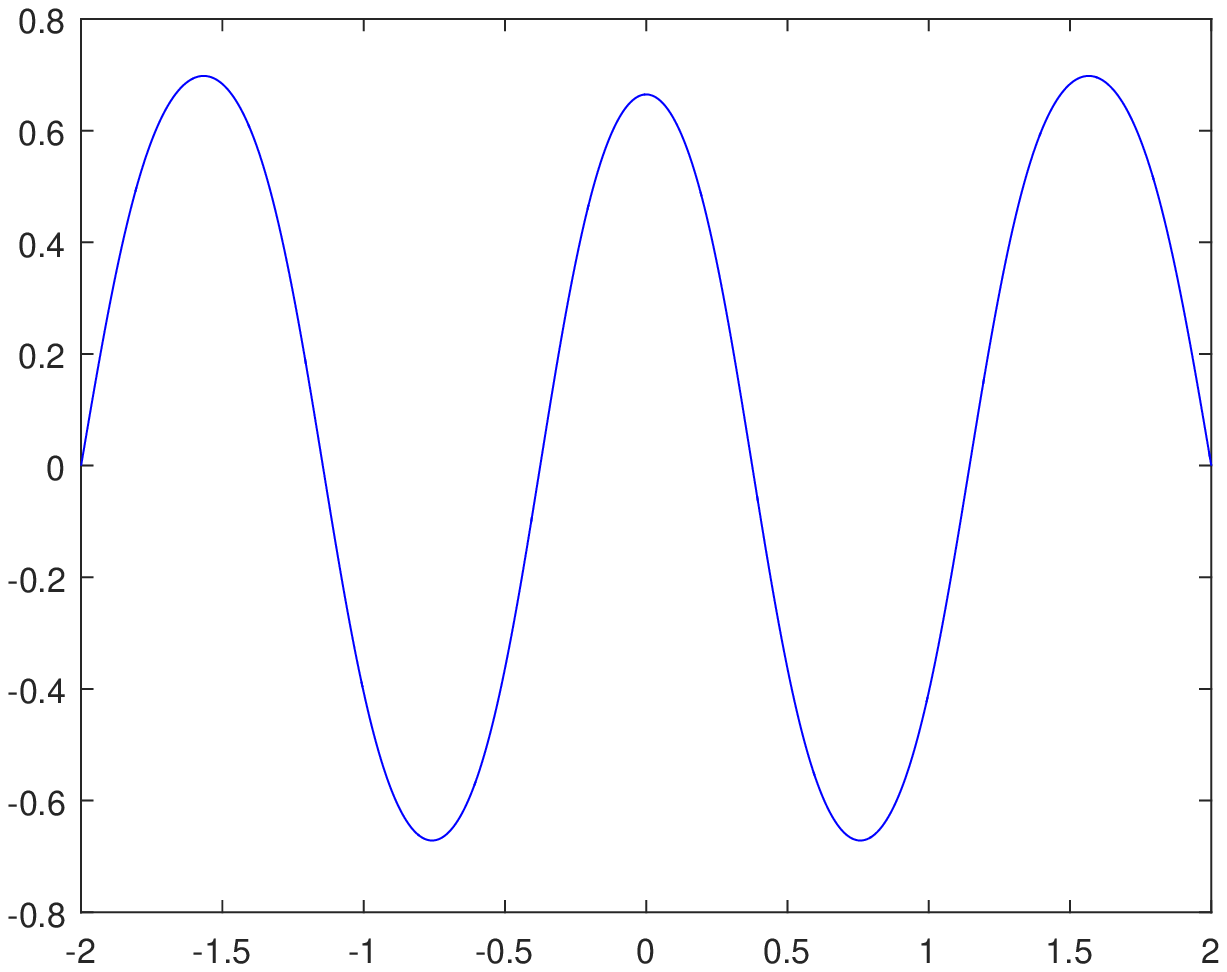}
          \caption{$\lambda = 15.62$}
          \end{minipage}
      }
      \subfigure
     {
         \begin{minipage}[b]{.3\linewidth}
         \centering
         \includegraphics[scale = 0.18]{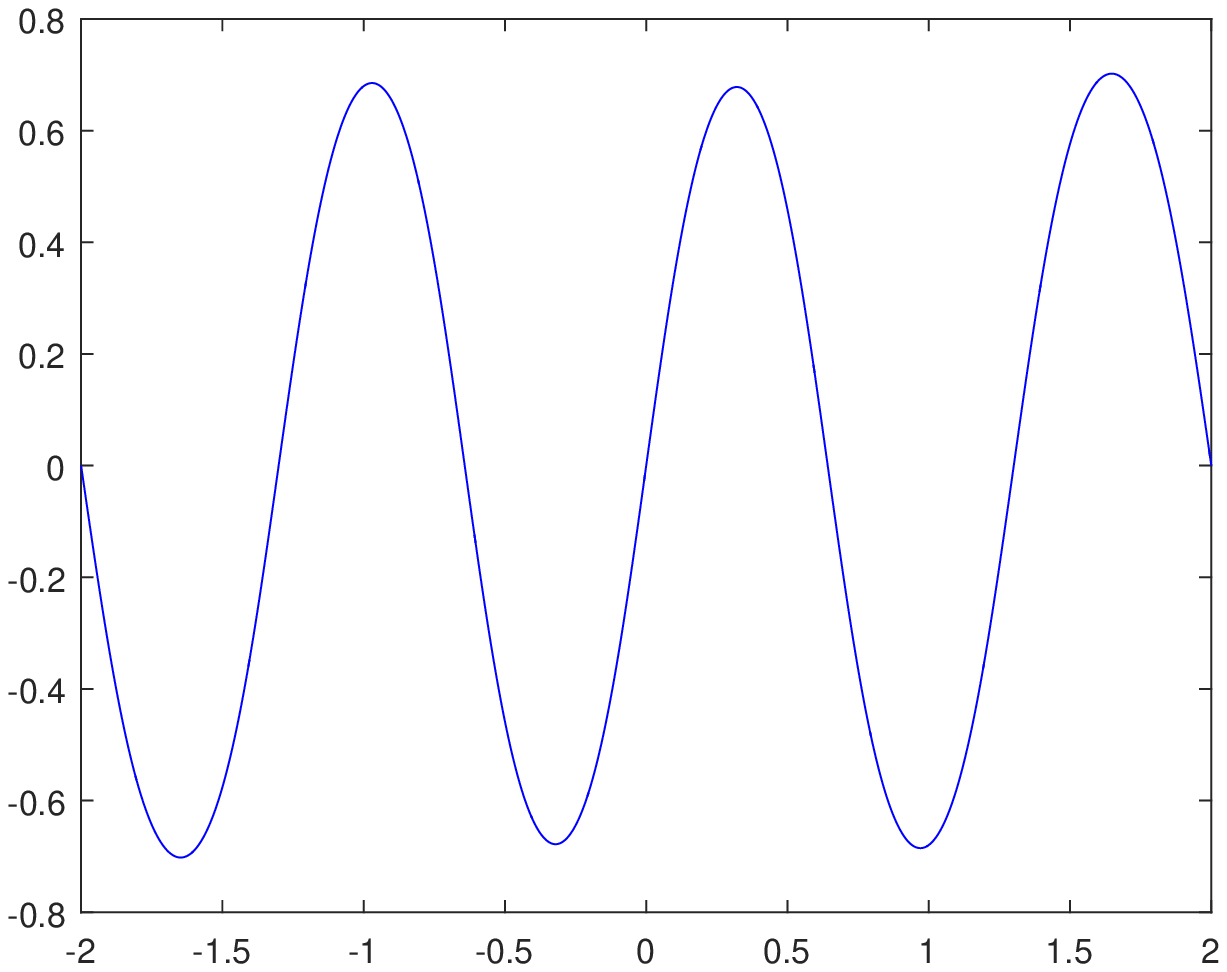}
         \caption{$\lambda = 19.08$}
         \end{minipage}
      }
\end{figure}
\vspace{-1cm}
\begin{figure}[htb]
     \centering
     \subfigure
     {
         \begin{minipage}[b]{.3\linewidth}
         \centering
         \includegraphics[scale = 0.18]{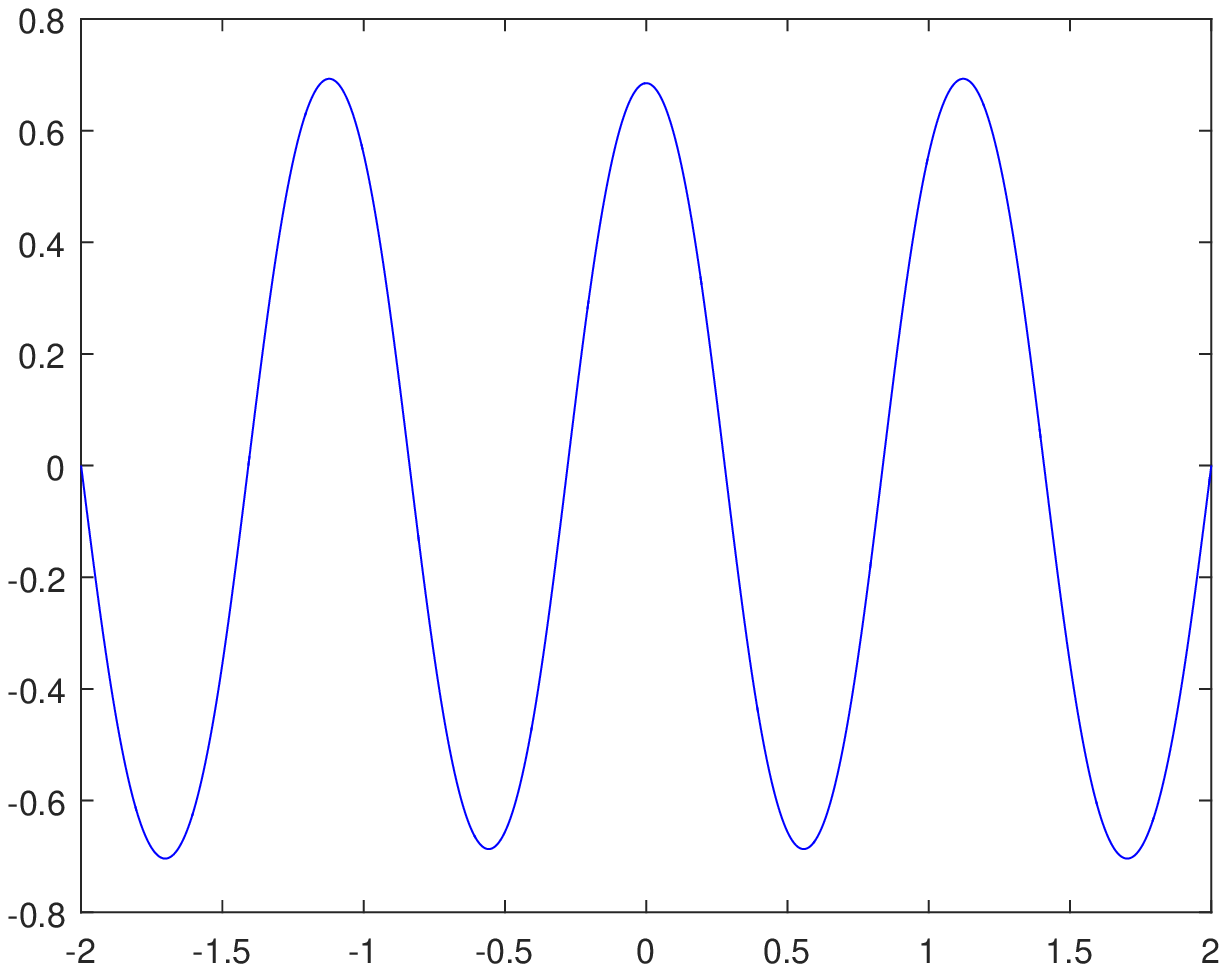}
          \caption{$\lambda = 23.13$}
          \end{minipage}
     }
    \subfigure
    {
        \begin{minipage}[b]{.3\linewidth}
         \centering
         \includegraphics[scale = 0.18]{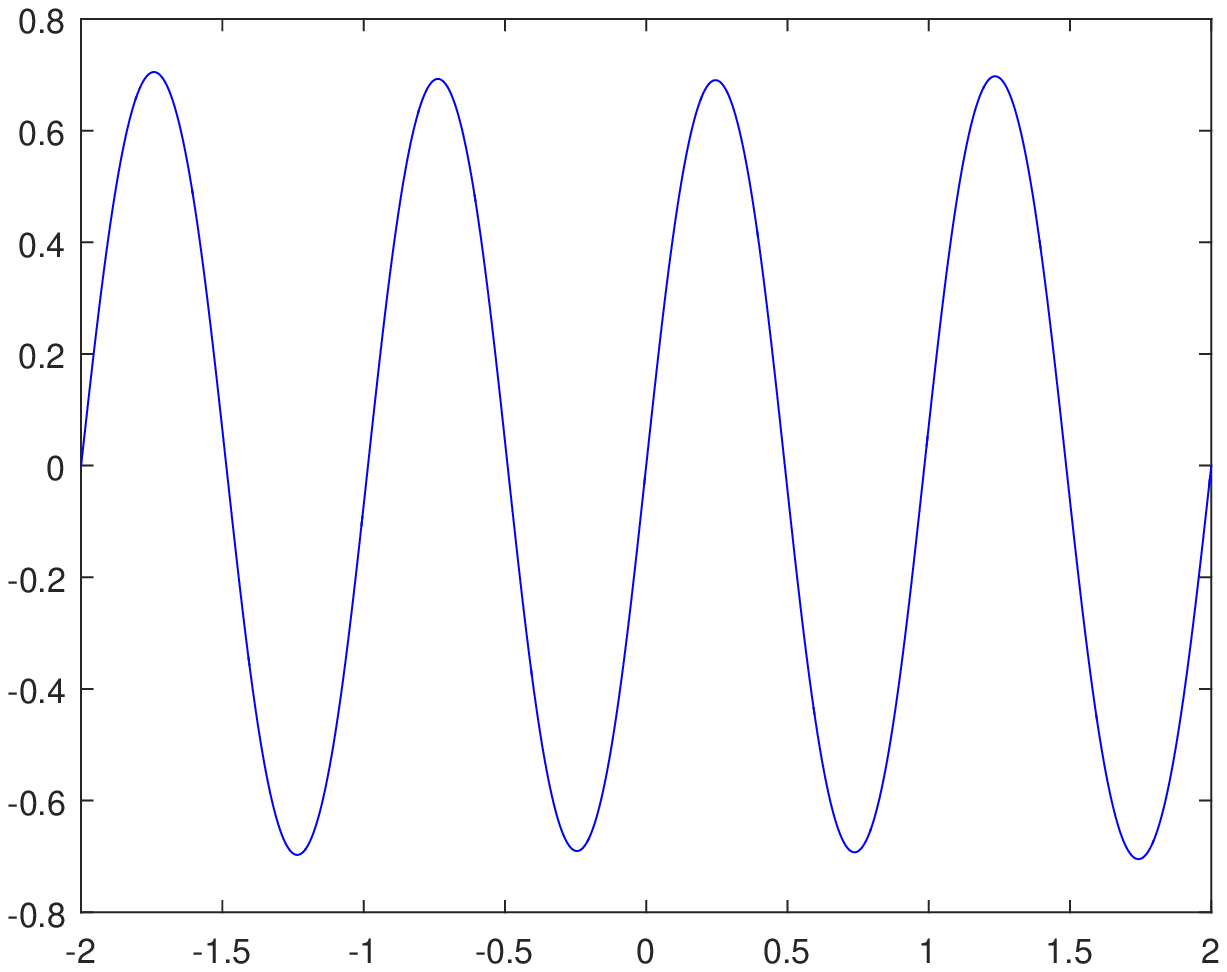}
         \caption{$\lambda = 27.79$}
         \end{minipage}
     }
    \subfigure
    {
       \begin{minipage}[b]{.3\linewidth}
       \centering
       \includegraphics[scale = 0.18]{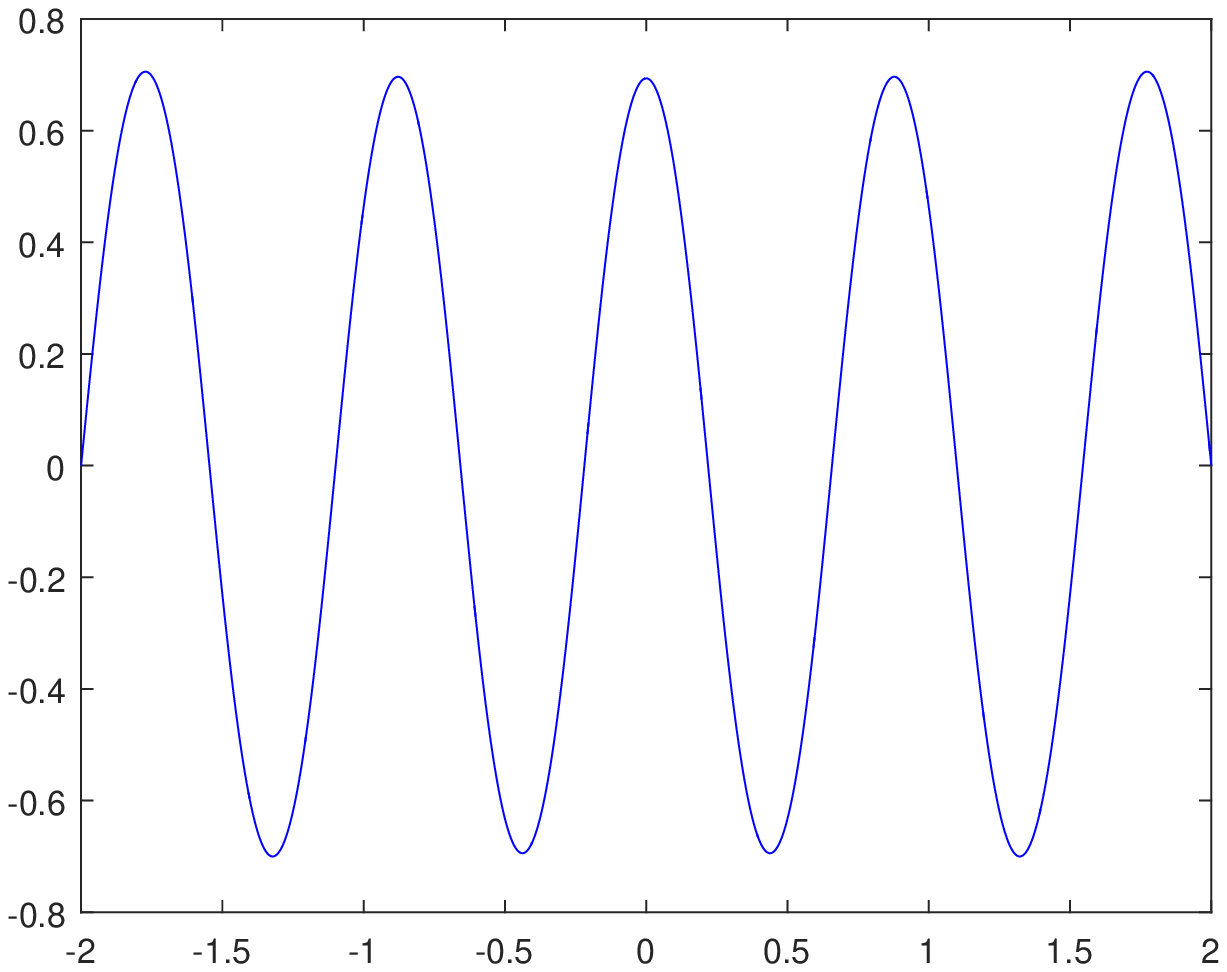}
        \caption{$\lambda = 33.05$}\label{sec4:fig-1D-9}
        \end{minipage}
     }
\end{figure}

The second numerical example is the 2D discretized problem $(\ref{sec 2: finite difference - two dimen})$ with $\beta = 20$, $V(x)=\frac{1}{2}(x_1^2 + x_2^2)$, $\Omega=[0,1]\times [0,1]$ and $m = n = 29$. Some approximate eigenfunctions are collected in Figures \ref{sec4:fig-2D-1} to \ref{sec4:fig-2D-12}. The approximate eigenfunction in Figure \ref{sec4:fig-2D-1} corresponds to the unique positive ground state. Others are approximate excited states corresponding to higher energy. The order preserving property of the eigenvalue curves is also observed for the 2D case.
\begin{figure}[htb]
     \centering
     \subfigure
     {
          \begin{minipage}[b]{.3\linewidth}
          \centering
          \includegraphics[scale = 0.25]{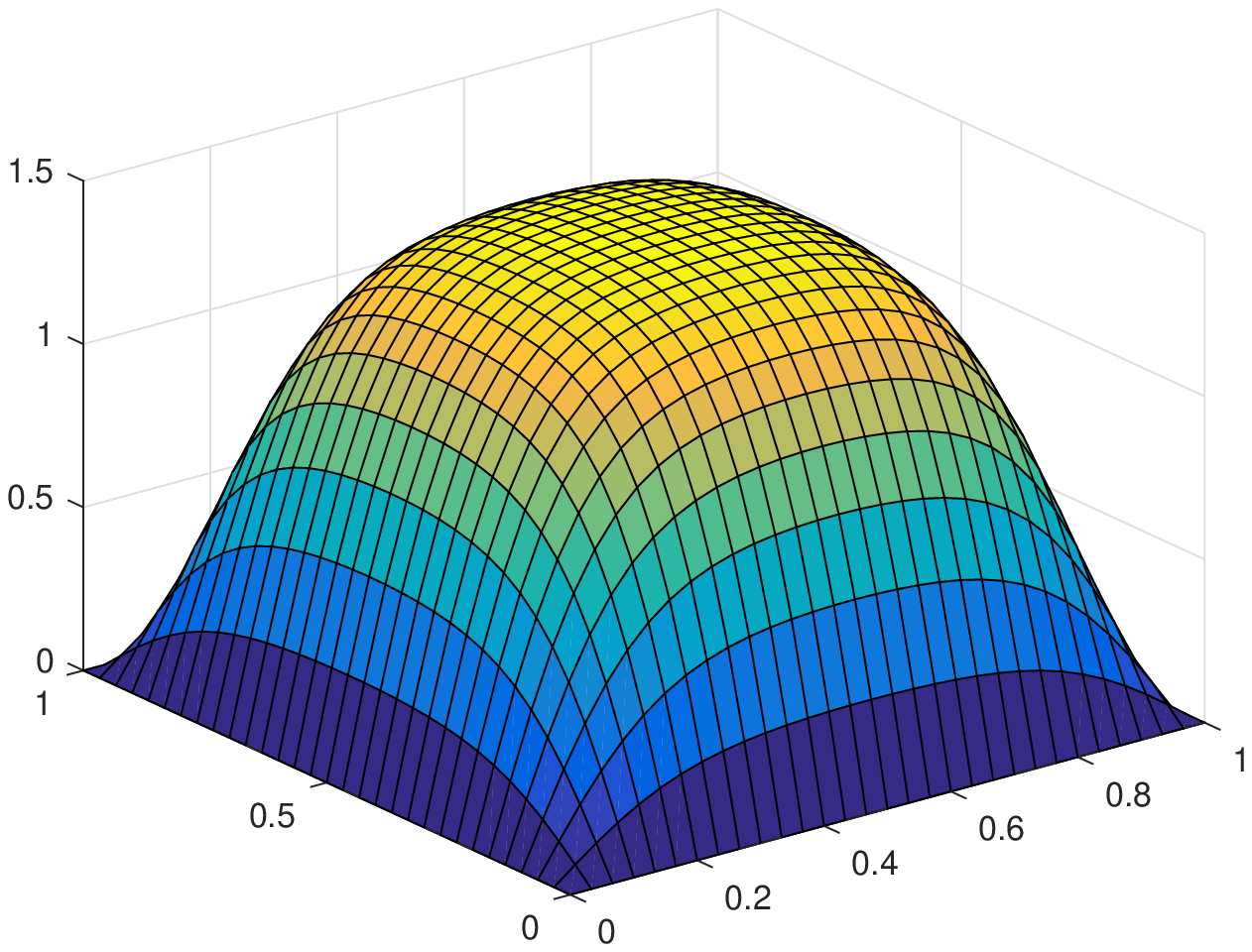}
          \caption{$\lambda  = 43.36$}\label{sec4:fig-2D-1}
          \end{minipage}
     }
     \subfigure
    {
        \begin{minipage}[b]{.3\linewidth}
        \centering
        \includegraphics[scale = 0.25]{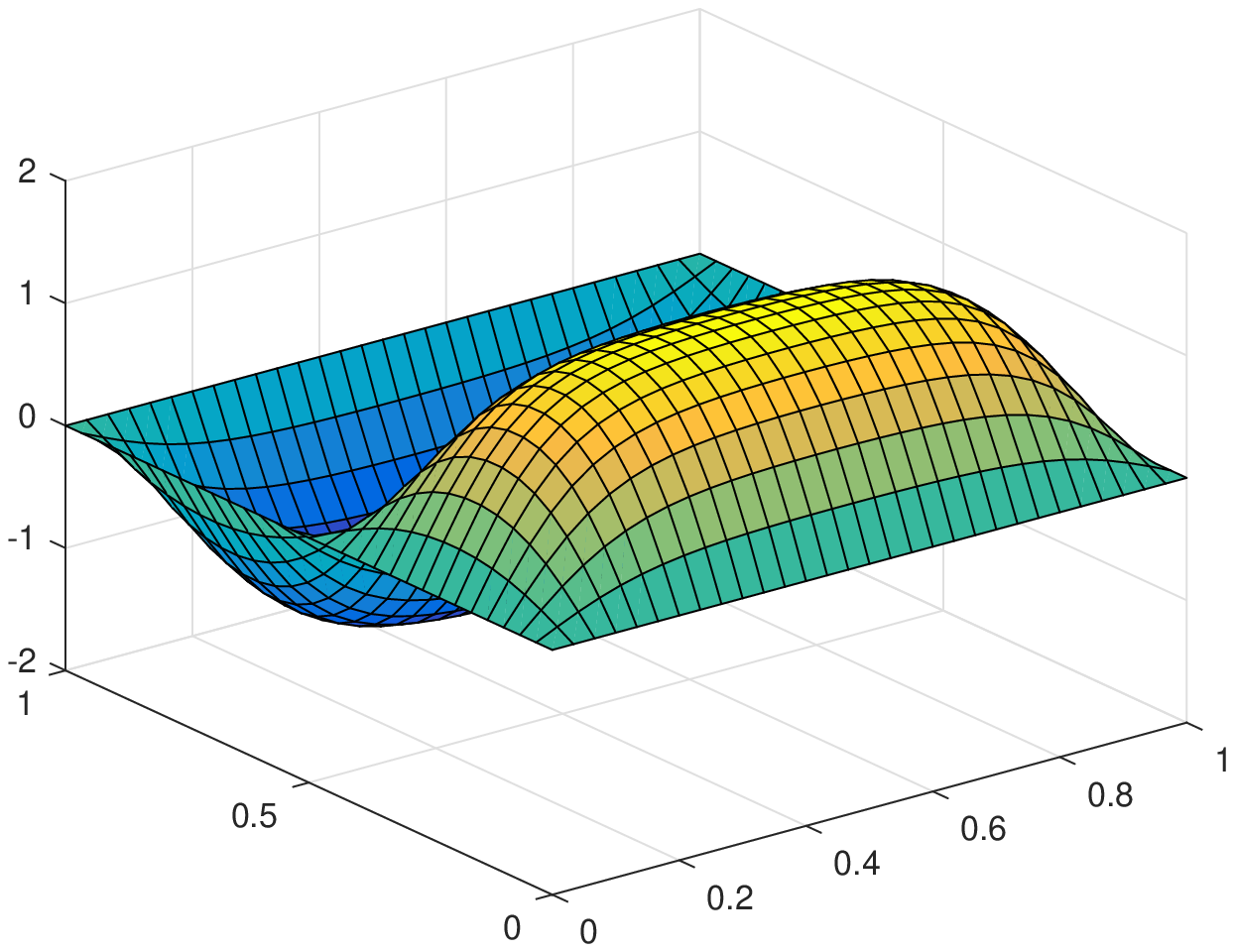}
        \caption{$\lambda  = 60.89$}
        \end{minipage}
    }
    \subfigure
   {
       \begin{minipage}[b]{.3\linewidth}
       \centering
       \includegraphics[scale = 0.25]{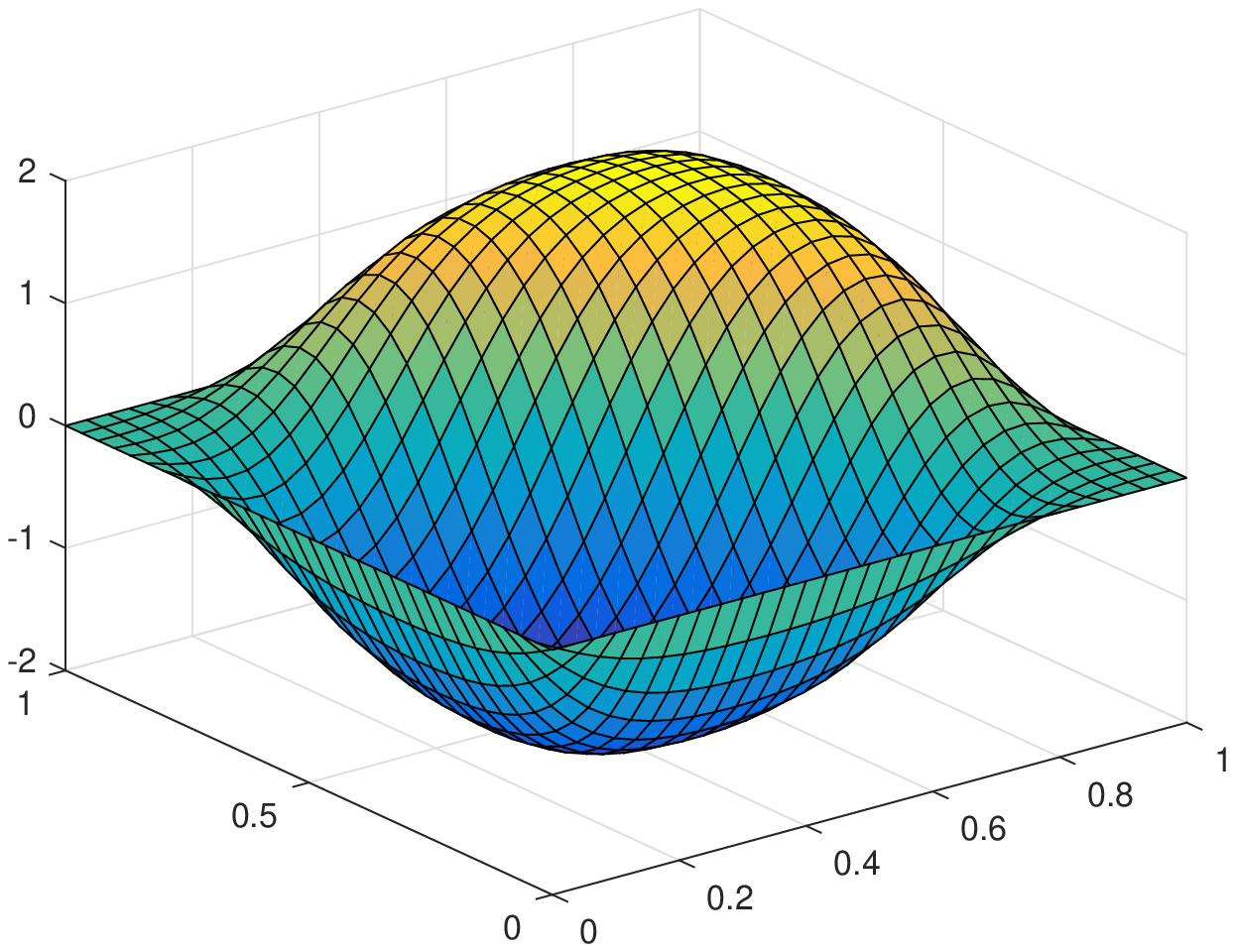}
        \caption{$\lambda = 65.31$}
        \end{minipage}
   }
\end{figure}
\vspace{1cm}
\begin{figure}[htb]
     \centering
     \subfigure
     {
         \begin{minipage}[b]{.3\linewidth}
         \centering
         \includegraphics[scale = 0.25]{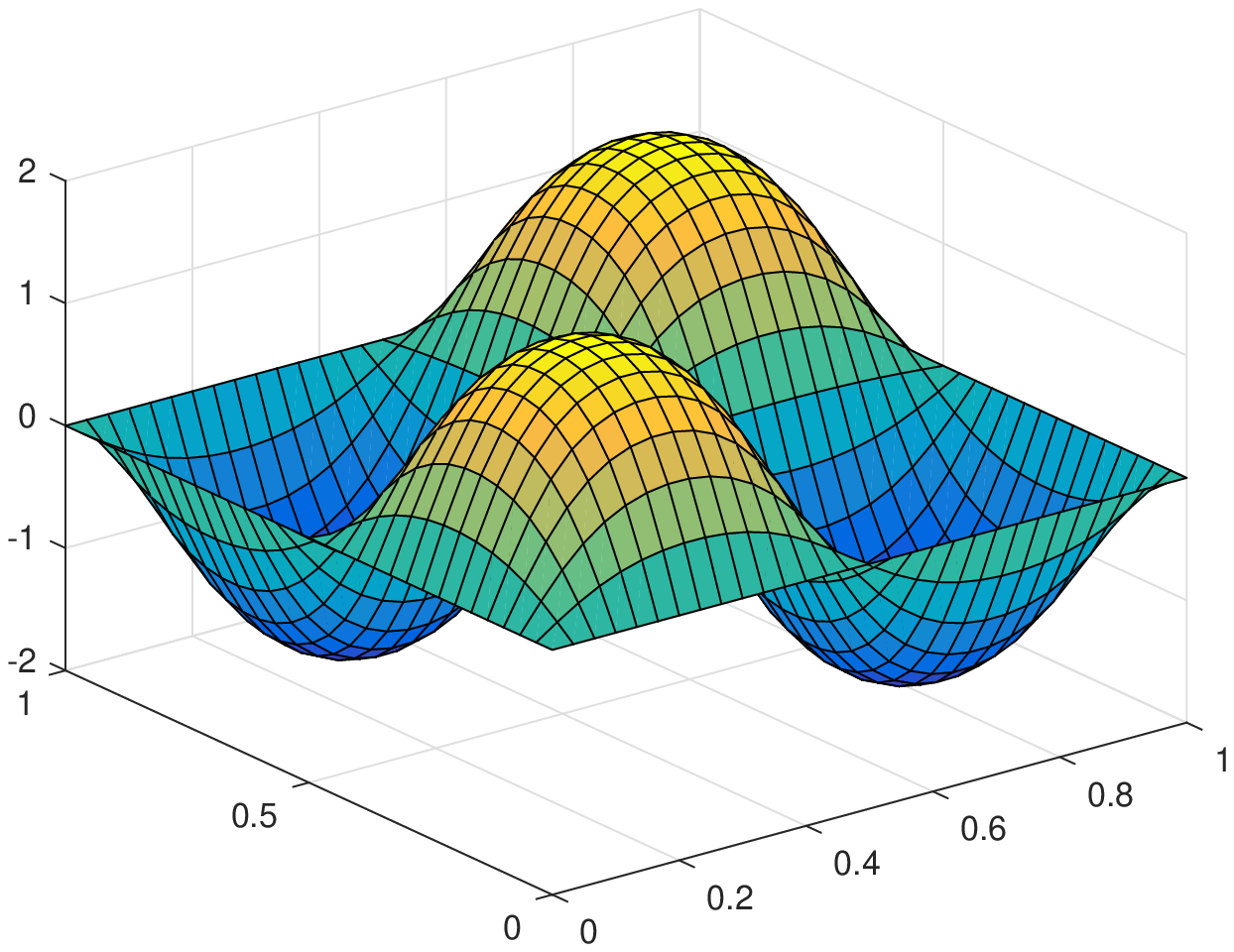}
         \caption{$\lambda = 78.81$}
        \end{minipage}
     }
    \subfigure
    {
        \begin{minipage}[b]{.3\linewidth}
        \centering
        \includegraphics[scale = 0.25]{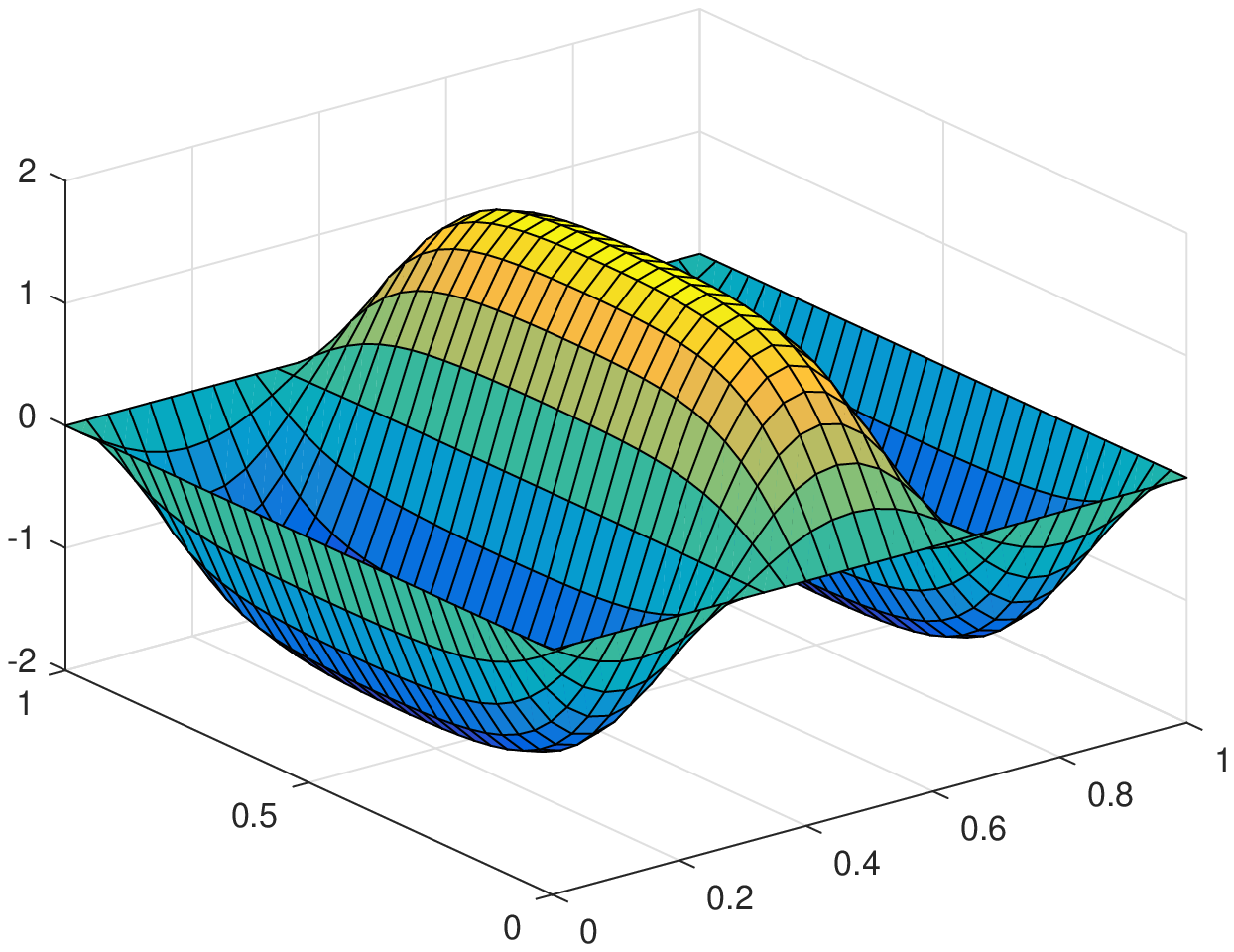}
        \caption{$\lambda = 86.28$}
        \end{minipage}
    }
    \subfigure
   {
        \begin{minipage}[b]{.3\linewidth}
        \centering
        \includegraphics[scale = 0.25]{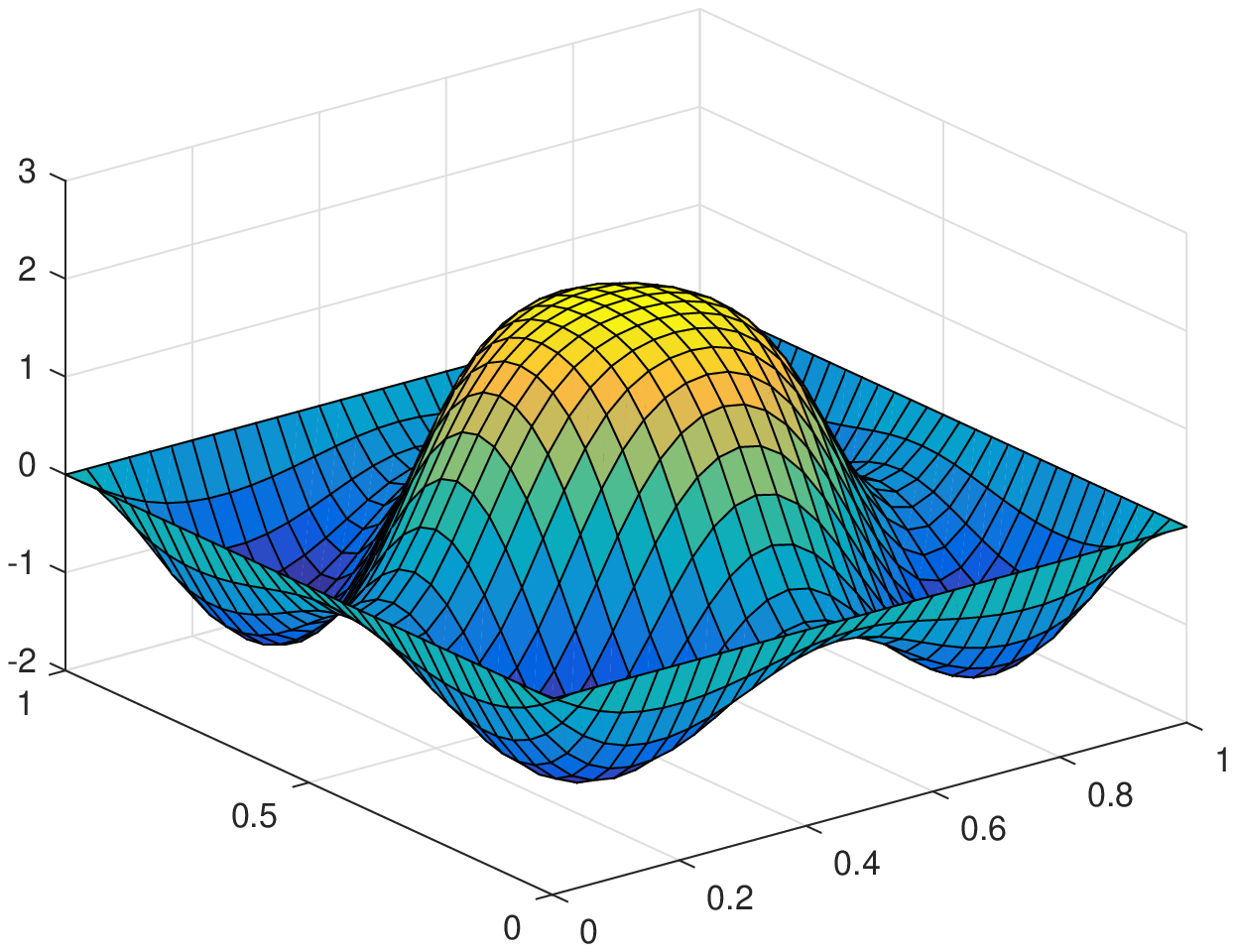}
        \caption{$\lambda = 94.06$}
        \end{minipage}
    }
\end{figure}
\vspace{-1cm}
\begin{figure}[htb]
     \centering
     \subfigure
     {
         \begin{minipage}[b]{.3\linewidth}
         \centering
         \includegraphics[scale = 0.25]{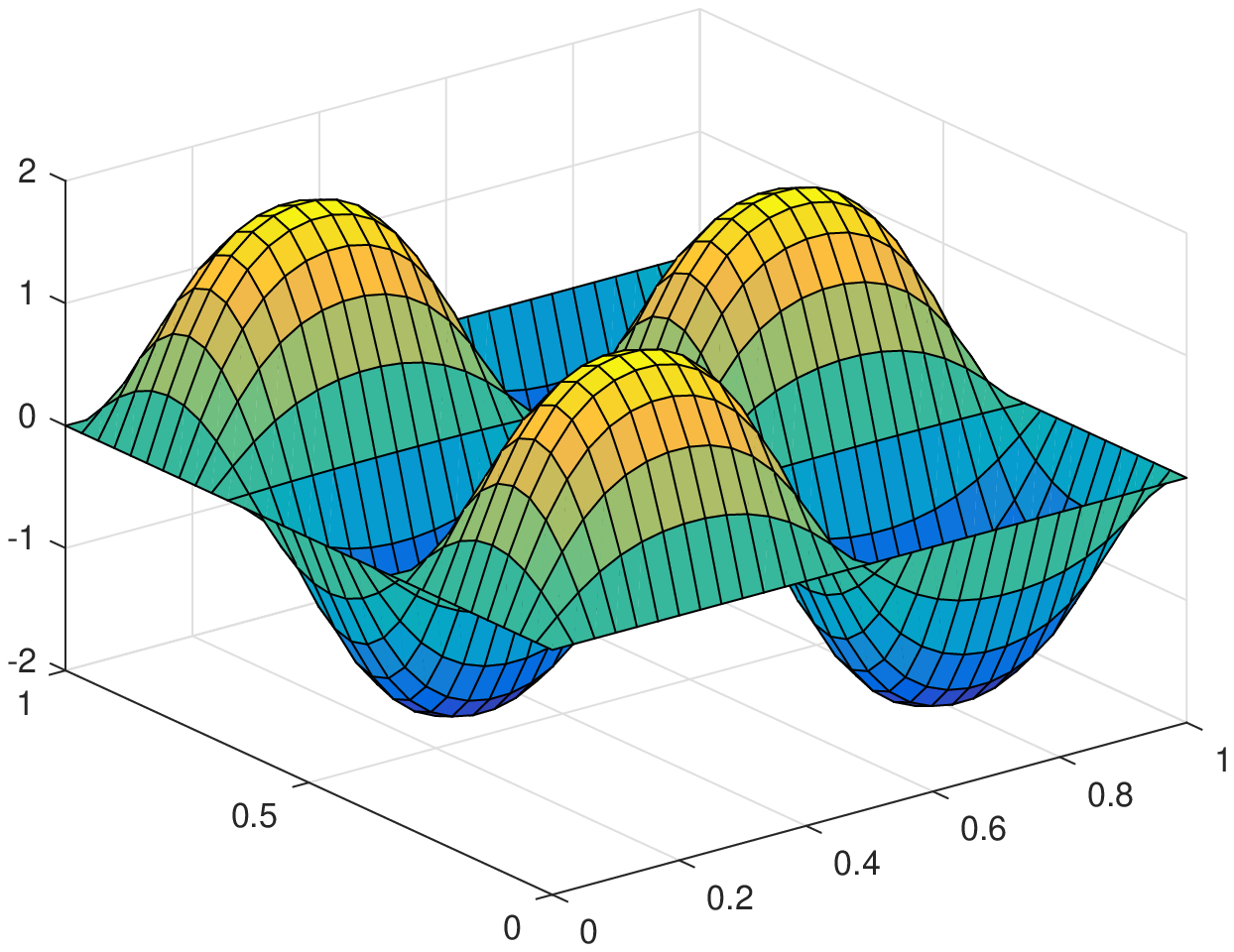}
         \caption{$\lambda = 104.4$}
         \end{minipage}
     }
    \subfigure
   {
        \begin{minipage}[b]{.3\linewidth}
        \centering
        \includegraphics[scale = 0.25]{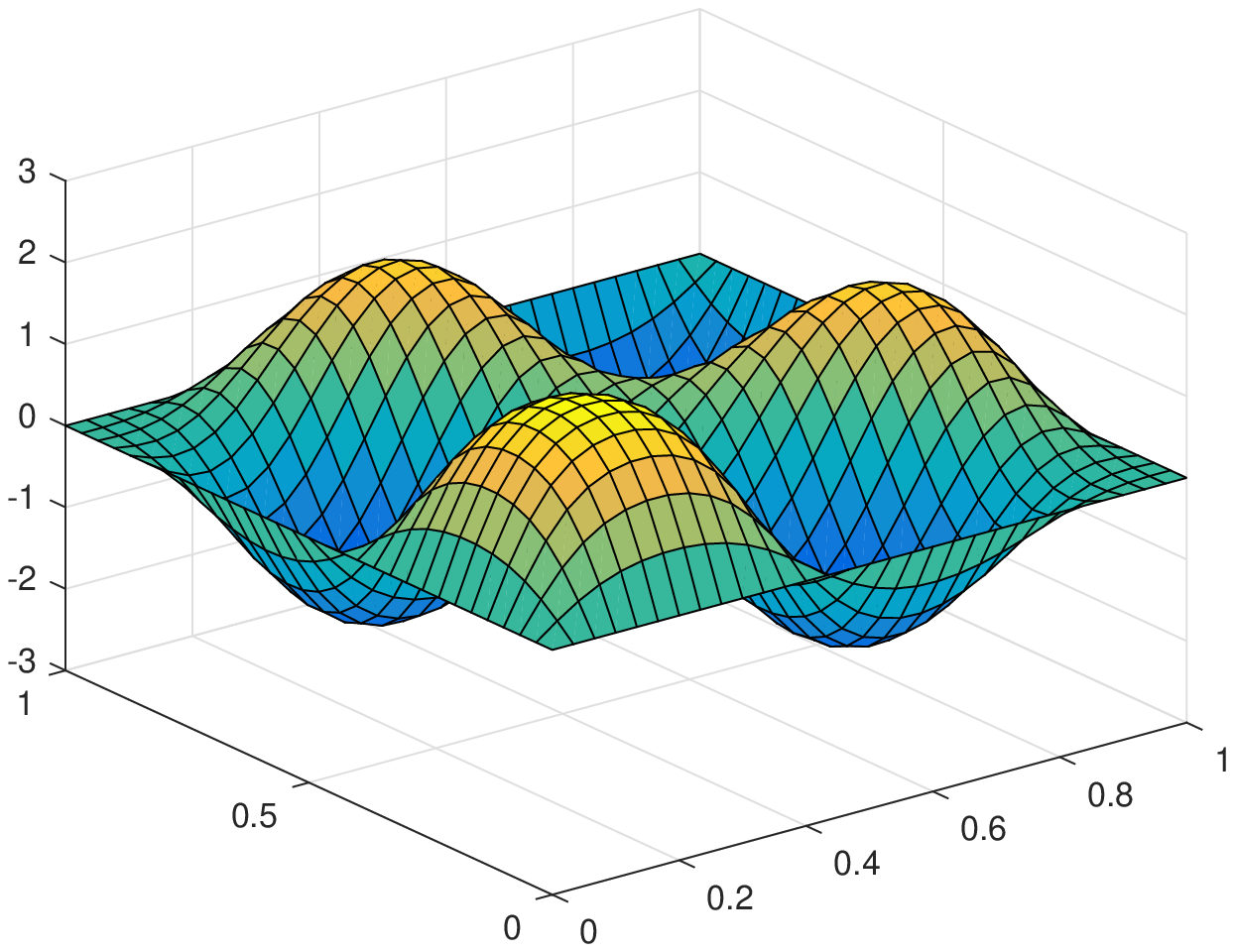}
        \caption{$\lambda = 108.82$}
        \end{minipage}
    }
    \subfigure
    {
        \begin{minipage}[b]{.3\linewidth}
        \centering
        \includegraphics[scale = 0.25]{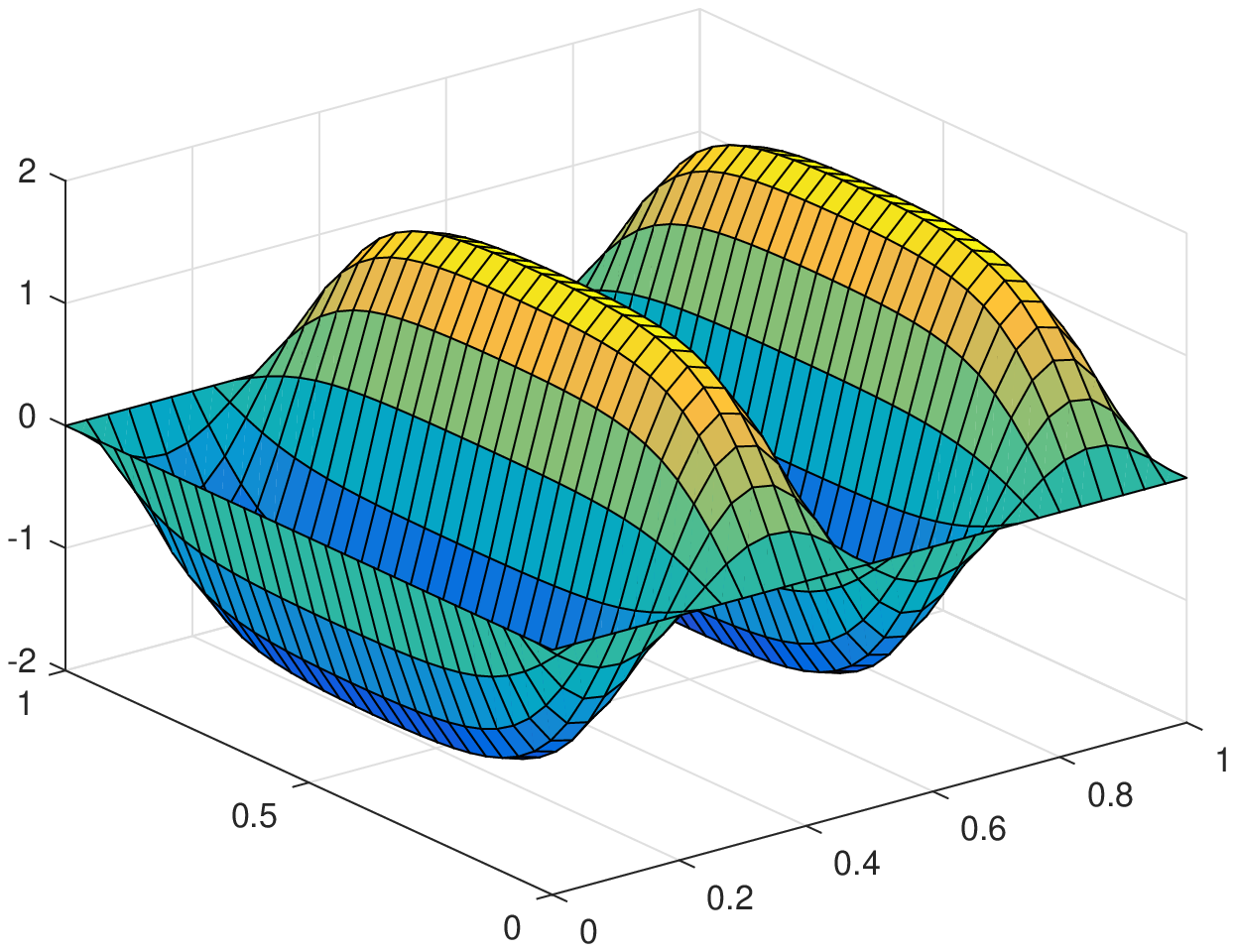}
        \caption{$\lambda = 120.47$}
        \end{minipage}
    }
\end{figure}
\vspace{-1cm}
\begin{figure}[H]
     \centering
     \subfigure
    {
        \begin{minipage}[b]{.3\linewidth}
        \centering
        \includegraphics[scale = 0.25]{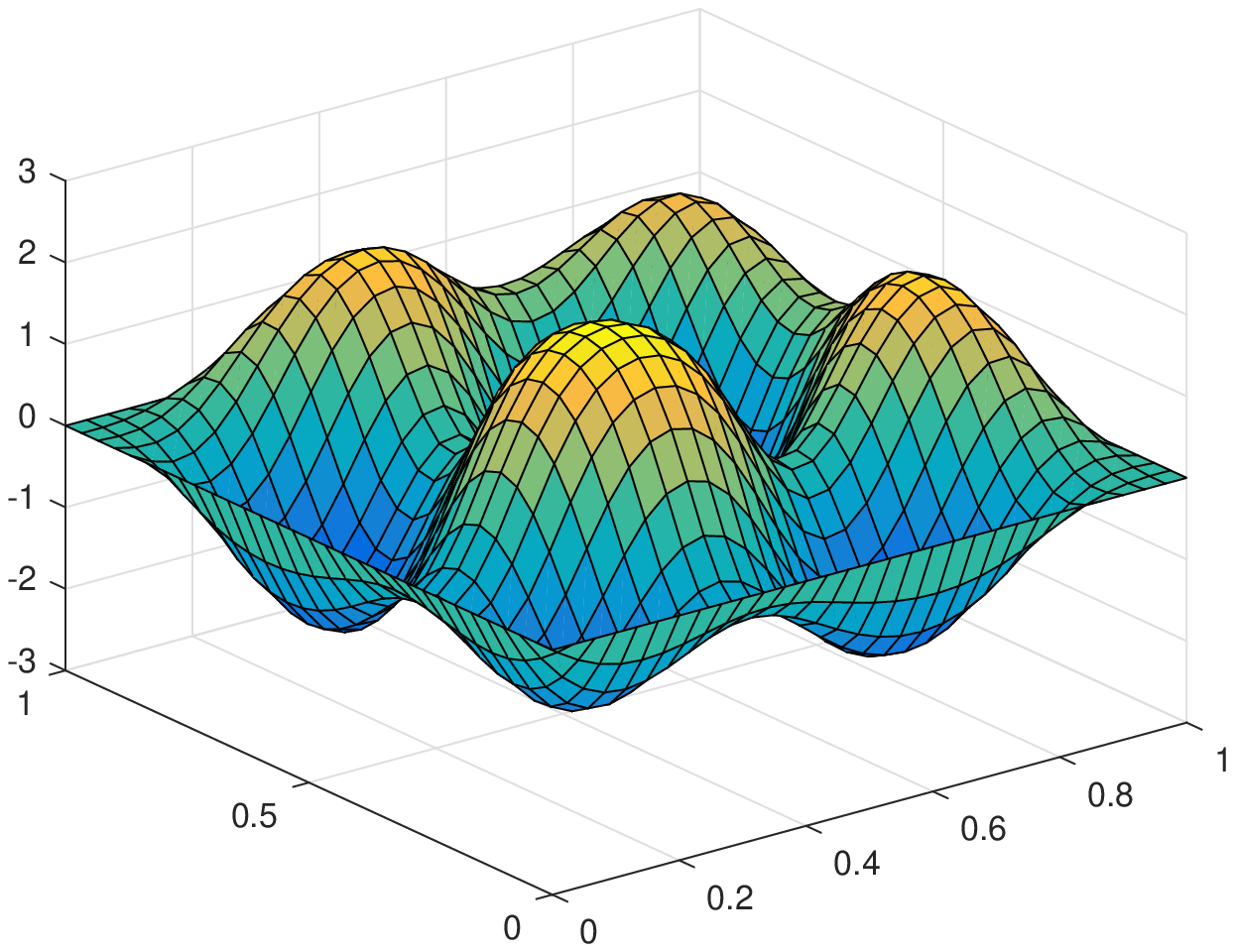}
        \caption{$\lambda = 129.67$}
        \end{minipage}
    }
    \subfigure
    {
        \begin{minipage}[b]{.3\linewidth}
        \centering
        \includegraphics[scale = 0.25]{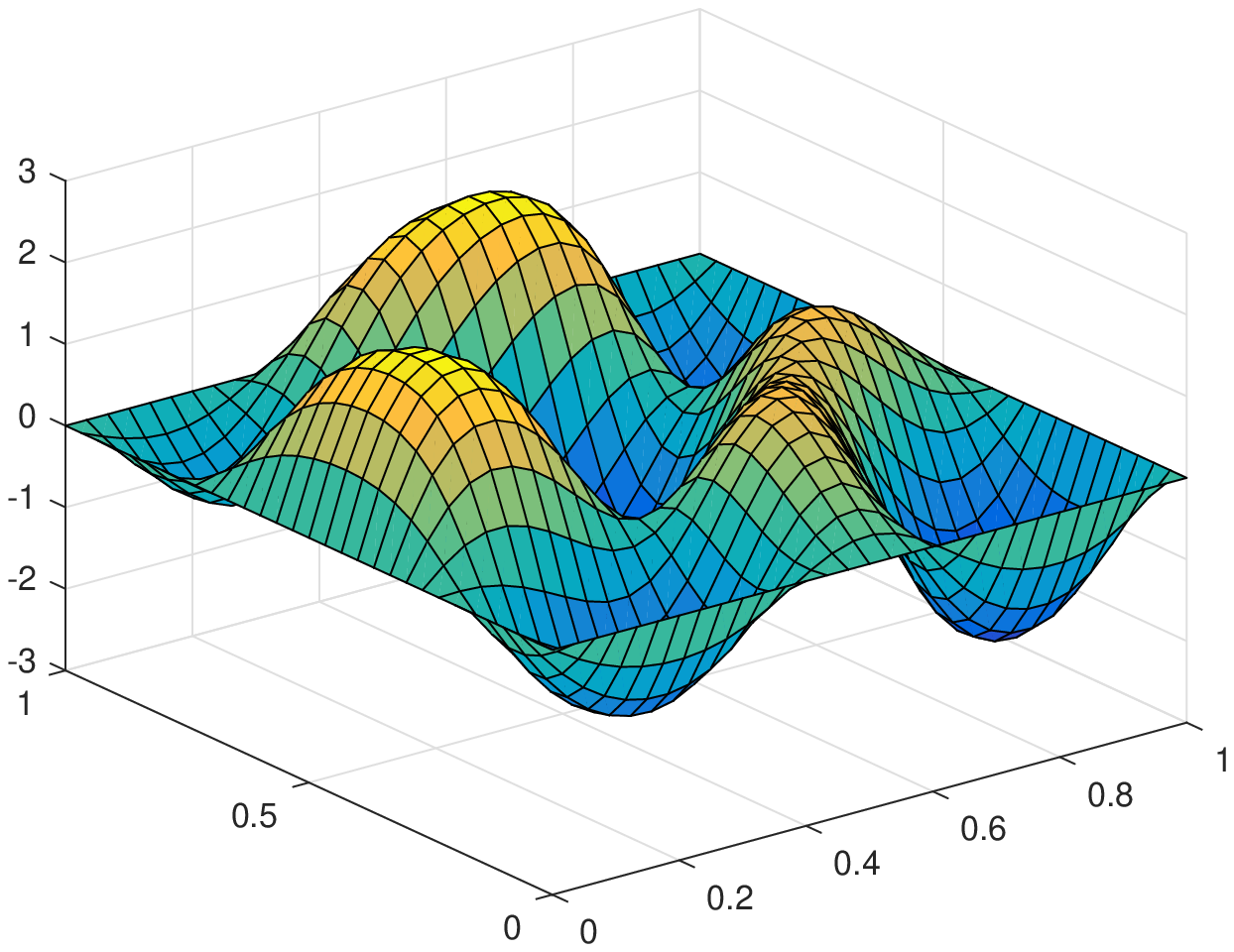}
        \caption{$\lambda = 133.81$}
        \end{minipage}
    }
    \subfigure
    {
        \begin{minipage}[b]{.3\linewidth}
        \centering
        \includegraphics[scale = 0.25]{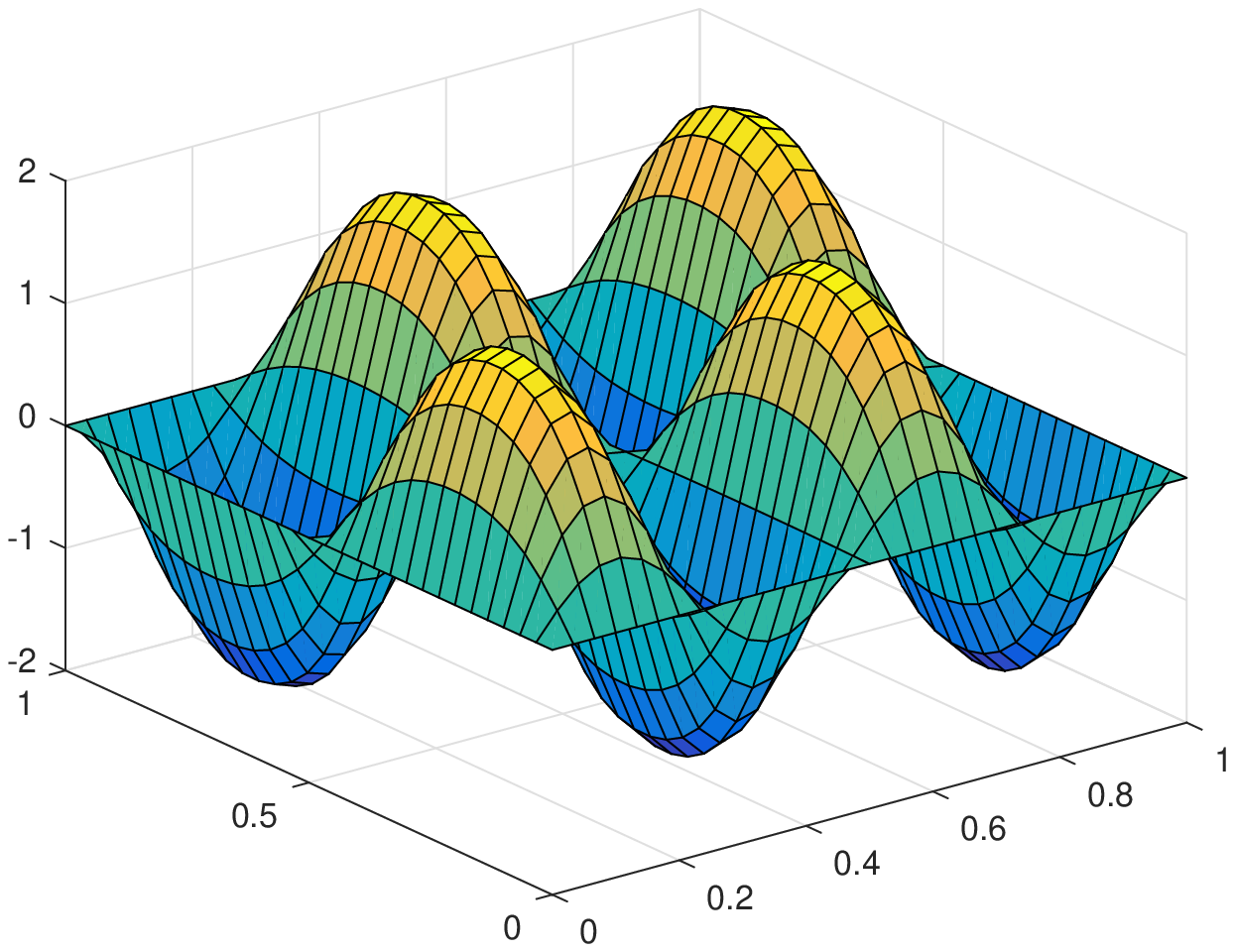}
        \caption{$\lambda = 138.68$}\label{sec4:fig-2D-12}
        \end{minipage}
    }
\end{figure}

\section{Conclusion}
Solutions to the discretized problem with the finite difference disretization for the GPE inherit certain properties of the solutions to the continuous problem, such as the existence and uniqueness of positive eigenvector (eigenfunction). The designed homotopy continuation methods are suitable for computing eigenpairs corresponding to excited states of high energy as well as the ground state and the first excited state. In order to make sure that the homotopy paths are regular and that the path following is efficient, artificial homotopy parameter and random matrices with certain structures in the homotopies seem indispensable.

\end{document}